\documentclass[sn-mathphys,Numbered]{sn-jnl}

\usepackage{graphicx}%
\usepackage{multirow}%
\usepackage{amsmath,amssymb,amsfonts}%
\usepackage{amsthm}%
\usepackage{mathrsfs}%
\usepackage[title]{appendix}%
\usepackage{xcolor}%
\usepackage{textcomp}%
\usepackage{manyfoot}%
\usepackage{booktabs}%
\usepackage{algorithm}%
\usepackage{algorithmicx}%
\usepackage{algpseudocode}%
\newtheorem{claim}{Claim}

\usepackage{listings}%

\usepackage{mathtools}
\usepackage{geometry}
\usepackage{todonotes}
\usepackage[algo2e]{algorithm2e} 
\usepackage[noend]{algpseudocode}
\usepackage{hyperref}
\usepackage{bm}
\usepackage{cases}
\usepackage{multirow}
\usepackage{comment}
\usepackage{xcolor}
\usepackage{pifont}
\usepackage{tikz}
\usepackage[noend]{algpseudocode}

\RequirePackage{xfrac}
\RequirePackage{braket} 

\theoremstyle{thmstyleone}%
\newtheorem{theorem}{Theorem}
\newtheorem{lemma}{Lemma}

\theoremstyle{thmstyletwo}%
\newtheorem{corollary}{Corollary}

\theoremstyle{thmstylethree}%
\newtheorem{definition}{Definition}%

\begin{document}

\title[Article Title]{Quantum Computing Inspired Iterative Refinement for Semidefinite Optimization}


\author[1]{\fnm{Mohammadhossein} \sur{Mohammadisiahroudi}}\email{mom219@lehigh.edu}

\author[2]{\fnm{Brandon} \sur{Augustino}}\email{baug@mit.edu}

\author[1]{\fnm{Pouya} \sur{Sampourmahani}}\email{pos220@lehigh.edu}

\author*[1]{\fnm{Tam\'as} \sur{Terlaky}}\email{terlaky@lehigh.edu}

\affil*[1]{\orgdiv{Department of Industrial and Systems Engineering}, \orgname{Lehigh University}, \orgaddress{\street{200 West Packer Avenue}, \city{Bethlehem}, \postcode{18015}, \state{PA}, \country{USA}}}

\affil[2]{\orgdiv{Sloan School of Management}, \orgname{Massachusetts Institute of Technology}, \orgaddress{\street{100 Main Street}, \city{Cambridge}, \postcode{02142}, \state{MA}, \country{USA}}}


\abstract{Iterative Refinement (IR) is a classical computing technique for obtaining highly precise solutions to linear systems of equations, as well as linear optimization problems. In this paper, motivated by the limited precision of quantum solvers, we develop the first IR scheme for solving semidefinite optimization (SDO) problems and explore two major impacts of the proposed IR scheme. First, we prove that the proposed IR scheme exhibits quadratic convergence toward an optimal solution without any assumption on problem characteristics. 
We also show that using IR with Quantum Interior Point Methods (QIPMs) leads to exponential improvements in the worst-case overall running time of QIPMs, compared to previous best-performing QIPMs. We also discuss how the proposed IR scheme can be used with classical inexact SDO solvers, such as classical inexact IPMs with conjugate gradient methods.}

\keywords{Iterative Refinement, Semidefinite Optimization, Quantum Interior Point Methods, Quadratic Convergence}



\maketitle

\section{Introduction}\label{intro}
Semidefinite optimization (SDO) constitutes a class of convex optimization problems that has a wide range of applications, however, its solution is computationally challenging. 
Letting $b \in \Rmbb^{m}$ and consider matrices $A_1, \dots, A_m, C \in \Scal^n$,
where $\Scal^n$ is the subspace of $n \times n$ symmetric
matrices. The primal SDO problem can be written in standard form as follows:
\begin{align}\label{e:SDO}\tag{P}
     z_P &= \inf_X \left\{ C \bullet X : A_i \bullet X = b_i~\text{for all } i \in [m], X \succeq 0 \right\},
\end{align}
where $[m]$ denotes the set $\{1, \dots, m\}$,  $U \bullet V$ denotes the trace inner product of symmetric matrices $U$ and $V$, i.e., $\trace{} (UV)$.
The notation $U \succeq V$ ($U \succ V$) indicates the \textit{L\"owner} order, i.e., $U - V$ is a symmetric positive semidefinite (positive definite) matrix. 
Without loss of generality, we assume the matrices $A_1, \dots, A_m$ are linearly independent. 
The dual problem associated with~\eqref{e:SDO} is 
\begin{align}\label{e:SDD}\tag{D}
    z_D &=  \sup_{y, S} \left\{ b^{\top} y:\sum_{i \in [m]} y_i A_i + S = C,~S\succeq 0, y \in \Rmbb^{m} \right\},
\end{align}
where $S \in \Scal^n$ denotes the slack matrix of the dual problem, i.e.,
\begin{equation*}\label{e:slackm}
  S = C- \sum_{i \in [m]} y_i A_i  \succeq 0. 
\end{equation*}
Linear optimization problems (LOPs) are a special case of semidefinite optimization problems (SDOPs) in which the input matrices are diagonal. 
Duality results for LO are stronger than those for SDO, except when there exist a primal feasible $X$ and dual feasible $(y, S)$ with $X, S \succ 0$, i.e., when the so-called \textit{Interior Point Condition} (IPC) holds. 
When the IPC is satisfied, optimal solutions exist for both the primal and dual SDO problems, and there is no duality gap, i.e., $z_P = z_D$.

Semidefinite optimization is routinely used to model problems in statistics, information theory \cite{rains2001semidefinite}, control \cite{boyd1994linear}, machine learning \cite{lanckriet2004learning, weinberger2006unsupervised}, as well as finance \cite{d2007direct, wolkowicz2012handbook}, and quantum information science \cite{aaronson2018shadow, eldar2003semidefinite, watrous2009semidefinite}. 
SDO has also been successfully applied to approximate the solution to combinatorial optimization problems, the most notable being the Goemans and Williamson SDO approximation of MaxCut \cite{goemans1995improved} and the Lov\'asz $\vartheta$-number \cite{lovasz1979shannon}.
In addition to LO, SDO generalizes other notable classes of convex optimization problems, such as Second-order Cone Optimization (SOCO) \cite{sam2024on}, and Quadratically Constrained Quadratic Optimization (QCQO) \cite{wolkowicz2012handbook}.
Further, SDOPs are known to exhibit
tractability \cite{alizadeh1998primal, nesterov1988general}, and can be used to study the properties of convex optimization problems \cite{boyd1994linear}. 

Interior Point Methods (IPMs) are the prevailing approach for solving SDOPs due to their polynomial complexity and rapid convergence \cite{nesterov1988general, nesterov1994interior}. 
However, IPMs for SDO pay a significant cost per iteration for large-scale problems; when $m = \Ocal \left(n^2\right)$ each iteration necessitates solving a Newton system of size $n^2 \times n^2$. 
Historically, IPMs relied on the use of direct methods solving linear systems of equations (e.g., Cholesky decomposition) to solve the Newton linear system. Consequently, the per-iteration cost scales as $\Ocal \left(n^6 \right)$, rendering the solution of large-scale SDO problems impractical.

These challenges, even before the birth of quantum computing, motivated the proposal of \textit{inexact} IPMs, in which the Newton linear systems are \textit{approximately} solved using iterative linear systems algorithms, such as the conjugate gradient method (CGM) \cite{bellavia2004convergence,zhou2004polynomiality}. 
Another type of Inexact IPM is ABIP method \cite{deng2022new} in which an ADMM method is used to solve the subproblems. 
Despite enjoying lower cost per iteration compared to exact IPMs, inexact IPMs are compromised by their inefficiency in achieving high precision solutions, particularly in their application to solving ill-conditioned or degenerate problems.
More recent developments for IPMs for SDO \cite{huang2022solving, jiang2020faster} seek to reduce the per-iteration cost by making use of fast matrix multiplication and data structures to efficiently maintain an approximation of the Hessian inverse. 
This approach allows one to \textit{amortize} the cost of exactly solving the Newton system over the run of the algorithm. Although, by making some additional assumptions,  these algorithms achieve state-of-the-art theoretical running times, to the best of our knowledge, they have not led to a successful practical implementation.

We also remark that IPMs are not the only approach for solving SDOPs.
Indeed, some researchers reformulated SDOPs as nonlinear optimization problems \cite{burer2003nonlinear} and employed first-order methods for their solution \cite{monteiro2003first, renegar2014efficient}. 
Another first-order method used for solving SDP is the Matrix-Multiplicative Weights Update method \cite{arora2006multiplicative}.

Since quantum computing emerged as a new paradigm for computing to speed up solutions to some mathematical problems, several efforts have been made to quantize optimization methods.
Quantum linear systems algorithms (QLSAs) \cite{harrow2009quantum} are a class of algorithms that use \textit{Hamiltonian simulation} to prepare a quantum state that is proportional to the solution of a linear system. 
These developments sparked the development of Quantum IPMs (QIPMs) \cite{augustino2023quantum, kerenidis2020quantum}, in which quantum subroutines are used to reduce cost per iteration in IPMs.

\subsection{Related work}
\textbf{Quantum interior point methods.} 
The first work on quantum IPMs (QIPMs) is attributed to Kerenidis and Prakash \cite{kerenidis2020quantum}, who presented QIPMs for LO and SDO. 
Their work sought to leverage QLSAs to accelerate the solution of the Newton linear system. 
The QLSA with block encodings can solve the Newton system in time polylogarithmic in $n$, albeit at a cost of linear dependence on the condition number. 
To proceed to the next iteration, a tomography subroutine is used to obtain a classical estimate of the solution obtained from applying a QLSA, which introduces additional overhead in the dimension and the error to which the estimation is performed.
Though their work provided a foundation for realizing QIPMs, critical challenges inherent to quantizing IPMs, and the subsequent consequences, were overlooked. 
In their QIPM, the Newton systems solved at each iteration are not correct, since there are no steps to guarantee symmetry. 
In addition, the complexity analysis of their QIPM did not properly account for the errors resulting from inexact tomography.

The first provably convergent QIPMs for SDO were developed by Augustino et al.~\cite{augustino2023quantum}, who both quantized an Infeasible-Inexact IPM and presented a novel Inexact-Feasible QIPM. 
In the first approach, they solved a Newton linear system where the right hand has residuals to capture the infeasibility and inexactness of the search direction. 
In the second approach, they recover a feasible IPM framework by using a nullspace representation of the search directions in the Newton linear system to ensure that the feasibility of the iterates is always guaranteed. 
The IF-QIPM can be shown to solve SDO problems involving $n \times n$ matrices and $m = \Ocal(n^2)$ constraints using at most\footnote{The $\widetilde{\Ocal}_{\alpha, \beta} \left( g(x) \right)$ notation indicates that quantities polylogarithmic in $\alpha, \beta$ and $g(x)$ are suppressed.}  
$$
   \widetilde{\Ocal}_{n, \kappa, \frac{1}{\epsilon}} \left(  n^{3.5}  \frac{\kappa^2}{\epsilon}  \right)
$$
accesses to a QRAM and $\widetilde{\Ocal}_{n, \kappa, \frac{1}{\epsilon}}\left(  n^{4.5}   \right)$ arithmetic operations, where $\kappa$ is an upper bound for the condition number of Newton systems, and $\epsilon$ is the target optimality gap.

The techniques introduced in \cite{augustino2023quantum} have subsequently been specialized to Linear Optimization \cite{Mohammadisiahroudi2023QIPM, mohammadisiahroudi2022efficient, mohammadisiahroudi2023inexact} and Linearly Constrained Quadratic Optimization \cite{wu2023inexact}. 
Huang et al.~\cite{huang2022faster} gave a QIPM for SDO by quantizing a robust dual-IPM framework. Their algoritm has analogous polynomial dependence of $\kappa$ and $\frac{1}{\epsilon}$. Observe, that the iterative refinement scheme we propose in this paper relies on both primal and dual information for progression between iterates, and it is unclear whether our framework can be adapted to a dual-only setting. 
Summarizing, without empirical evidence to suggest otherwise, the dependence on $\kappa$ and $\frac{1}{\epsilon}$ suggests that the speedups in $n$ obtained by QIPMs for SDO are insufficient to obtain an overall speedup over the classical algorithm. 
Rather, it would appear that QIPMs are exponentially slower than their classical counterpart, whose complexity does not depend on $\kappa$, and only logarithmically depend on $\frac{1}{\epsilon}$. 

Very recently, Apers and Gribling \cite{apers2023quantum} proposed a QIPM for LO that avoids dependence on a condition number. 
Under some mild assumptions and access to QRAM, their framework achieves a quantum speedup for ``tall'' LOPs, in which the number of constraints is much larger than the number of variables (i.e., $m \gg n$). 
Rather than using QLSAs to solve the Newton system, the Newton steps are computed via spectral approximations of the Hessian.
While it is an interesting open question whether the techniques from \cite{apers2023quantum} can be generalized to SDO, it has been posited in the classical literature that sampling is unlikely to speed up Hessian computation as it does in IPMs for LO \cite{jiang2020faster}. 
Therefore, we seek another avenue for alleviating the condition of number dependence. 

In the papers \cite{mohammadisiahroudi2022efficient, mohammadisiahroudi2023inexact}, the authors used an iterative refinement (IR) method to avoid an exponential time of finding an exact optimal solution. 
The basic idea in the context of QIPMs can be understood as follows: rather than solve the optimization problem at hand using a QIPM run to high precision, the IR algorithm treats the QIPM as a low-precision oracle to solve a sequence of optimization problems to constant accuracy and use these intermediate solutions to construct a high-precision solution. 

\textbf{Iterative refinement.} Iterative refinement stands as a widely adopted technique for enhancing numerical accuracy in the resolution of linear systems of equations \cite{wilkinson2023rounding, golub2013matrix}. 
Pioneering this approach for LO problems, Gleixner et al. \cite{Gleixner2016_Iterative} introduced the first iterative refinement methodology for constrained optimization problems. They applied this technique to solve LO problems precisely, while leveraging limited-precision oracles \cite{Gleixner2020_Linear}. 
Additionally, Mohammadisiahroudi et al.~\cite{mohammadisiahroudi2022efficient} utilized IR methods to obtain exact solutions for LO problems in the context of limited-precision QIPMs. 
Their work demonstrated that IR can partially alleviate the impact of ill-conditioned Newton systems on the computational complexity of QIPMs.
Building upon these foundations, our paper introduces a novel extension of this concept. 
We propose the first IR method tailored specifically for SDOPs. 
Notably, SDO presents a computationally more demanding scenario compared to LO and linear system problems. 
Through this extension, we aim to explore the potential of IR in mitigating challenges in the SDO domain.

\textbf{Convergence rates.} Analysis of convergence of IPMs for SDO is more complicated than those of LO \cite{ji1999local}. 
It is a challenge to show the superlinear convergence of an IPM with polynomial iteration complexity for SDO \cite{sim2023superlinear}. 
The first papers that demonstrated superlinear convergence of an IPM for SDO were by Kojima et al. \cite{kojima1998local} and Potra and Sheng \cite{potra1998superlinearly}. 
However, the former had to make certain key assumptions including nondegeneracy, strict complementarity, and tangential convergence to the central path for their proofs, and the latter proposed a sufficient condition to establish the result, both for an infeasible IPM. 
Later, Luo et al.~\cite{luo1998superlinear} proved superlinear convergence without nondegeneracy assumption.
Potra and Sheng~\cite{potra1998superlinearly} showed that quadratic convergence can be achieved under further assumptions. 
Later, Kojima et al.~\cite{kojima1999predictor} proved the quadratic convergence for an infeasible IPM using Alizadeh-Haeberly-Overton (AHO) direction. 
However, showing the same result for other directions, such as the Nesterov-Todd (NT) direction, is not easy. 
To show the superlinear convergence for these search directions, in addition to the assumptions mentioned earlier, modifications to the algorithm, such as solving the corrector-step linear system in an iteration repeatedly instead of only once, are required \cite{sim2023superlinear}.
We show that the proposed IR method has quadratic convergence toward the optimal solution set, even degenerate problems, and/or for problems that do not have  strictly complementary solutions.

\subsection{Contributions of this paper}\label{contrib}
The contributions of this paper are outlined as follows:
\begin{itemize}
\item \textbf{Introduction of Iterative Refinement Methods for SDO:} We present the first Iterative Refinement methods designed for finding precise solutions of SDO problems utilizing limited-precision oracles. This method holds the potential to bridge the gap between exact IPMs with computationally expensive iterations and inexact SDO solvers characterized by slow convergence.

\item \textbf{Quadratic Convergence of the Proposed IR Method:}
We demonstrate that the proposed IR method exhibits quadratic convergence towards the optimal solution set of the SDO problem. This noteworthy result holds even in cases where strict complementarity fails, yielding the first quadratically convergent IPM and QIPM for SDO. 

\item \textbf{Development of IR-IF-QIPM for SDO:}
We introduce an Iteratively Refined Inexact Feasible Quantum Interior Point Method (IR-IF-QIPM) for SDO. This novel approach demonstrates improved computational complexity concerning dimension compared to classical IPMs and significantly improved precision dependence compared to previous QIPMs. To obtain inexact but feasible Newton directions in the proposed IR-IF-QIPM, we used the recently proposed Orthogonal Subspaces System \cite{augustino2023quantum, mohammadisiahroudi2023inexact}. Furthermore, we illustrate how IR can effectively mitigate the impact of ill-conditioned Newton systems on the complexity of QIPMs.

\item \textbf{Variants of the Proposed IR Method:} Different versions of the proposed IR method are tailored for scenarios where the limited-precision oracle produces solutions with limited infeasibility errors. Additionally, the classical counterpart of the proposed IR-IF-QIPM is presented in Appendix~\ref{sec: CGM}, utilizing CGMs to inexactly determine the Newton direction.
\end{itemize}

The rest of this paper is structured as follows. 
In Section~\ref{sec: prelim}, we review some essential preliminary concepts pertinent to IPMs for SDO. 
We present our IR method for SDO in Section~\ref{sec: IR}. 
Section~\ref{sec: convergence} is dedicated to the convergence analysis of the proposed IR method.
Section~\ref{sec: qipm} presents the proposed IR-IF-QIPM designed for solving SDO problems.
Section~\ref{sec: conclusion} concludes the paper.
In addition, infeasible variants of the proposed IR are presented in Appendix~\ref{sec: otheIR}, and the classical version, an IR-IF-IPM using CGM is developed in Appendix~\ref{sec: CGM}.

\section{Preliminaries}\label{sec: prelim}
We denote the quantity $a$ to the $k$-th power by $a^{(k)}$, and the notation $a^{(k)}$, indicates the value of $a$ at iteration $k$ of an algorithm.

In what follows, we make extensive use of the matrix functions \textbf{svec} and \textbf{smat}, which are formally defined as follows. 
\begin{definition}
For $U \in \R{n \times n}$, $\textbf{\textup{svec}} (U) \in \R{\frac{1}{2} n (n+1)}$ is given by
    $$ \textup{\textbf{svec}}(U) =\left(u_{11}, \sqrt{2} u_{21}, \dots, \sqrt{2} u_{n1}, u_{22}, \sqrt{2} u_{32} \dots, \sqrt{2} u_{n2}, \dots, u_{nn} \right)^{\top} .$$
Further, the operator $\textup{\textbf{smat}}$ is the inverse operator of $\textbf{\textup{svec}}$. That is, 
$$ \textbf{\textup{smat}} \left[ \textbf{\textup{svec}}  (U) \right] = U. $$
\end{definition}
\noindent The definition of the \textbf{svec} operator gives rise to the definition of the \textit{symmetric} Kronecker product, which we define using the usual Kronecker product. 

\begin{definition}
Consider an $\frac{n(n+1)}{2} \times n^2$ matrix $U$, with 
$$ U_{(i,j), (k,l)} = \begin{cases}
1 &\text{if}~i = j = k = l, \\
1/\sqrt{2} &\text{if}~i = j \neq k = l,~\text{or}~i = l \neq j = k, \\
0 &\text{otherwise}.
\end{cases}$$
Then, the symmetric Kronecker product $G \otimes_s K$ of $G$ and $K$ can be expressed in terms of the standard Kronecker products of $G$ and $K$ as:
$$G \otimes_s K = \frac{1}{2} U \left( G \otimes K + K \otimes G \right) U^{\top}.$$
\end{definition}

\subsection*{Computational cost}
When we refer to complexity, we mean the number of queries to the controlled version of the input oracles and their inverses. We define $\Ocal (\cdot)$ as
$$f(x) = \Ocal(g(x)) \iff \text{ there exist } \gamma \in \R{}, \  c \in {\R{}}_+~\text{such that}~f(x) \leq c g(x) \; \forall \, x > \gamma.$$
We write $f(x) = \Omega (g(x)) \iff g(x) = \Ocal(f(x))$, and we also define $\widetilde{\Ocal} (f(x)) =\Ocal(f(x) \cdot \textup{polylog}(f(x)))$. When the function depends polylogarithmically on other variables, we write $\widetilde{\Ocal}_{a, b} (f(x)) =\Ocal(f(x) \cdot \textup{polylog}(a, b, f(x)))$.

\subsection{Primal and Dual SDOs and the Central Path}\label{s:CentralPath}

Recall that we assume the matrices $A_1, \dots, A_m$ are linearly independent. We define the feasible sets of \eqref{e:SDO} and \eqref{e:SDD} to be
\begin{align*}
    \Pcal &= \left\{X \in \Scal^n~: A_i \bullet X = b_i \text{ for} ~i \in [m], \  X \succeq 0 \right\}, \\
    \Dcal &=  \left\{(y, S) \in \R{m} \times \Scal^n: \sum_{i \in [m]} y_i A_i + S = C, \  S \succeq 0 \right\},
\end{align*}
and the sets of \textit{interior feasible solutions} are given by
\begin{align*}
    \Pcal^{\circ} &= \left\{ X \in \Scal^n ~:~ A_i \bullet X = b_i \text{ for} ~i \in [m], \ X \succ 0 \right\},\\
    \Dcal^{\circ} &= \left\{ (y, S) \in \R{m} \times \Scal^n : \sum_{i \in [m]} y_i A_i + S = C, \  S \succ 0 \right\}.
\end{align*}

Feasible IPMs are predicated on the existence of a strictly feasible primal-dual pair $X \in \Pcal^{\circ}$ and $(y, S) \in \Dcal^{\circ}$. The existence of a strictly feasible initial solution ensures that the IPC is satisfied \cite{de1997initialization}, guaranteeing that the primal and dual optimal sets
\begin{align*}
    \Pcal^* &= \left\{X \in \Pcal :  C \bullet X = z_{P^*} \right\}, \\
    \Dcal^* &=  \left\{(y, S) \in \Dcal :  b^{\top} y = z_{D^*} \right\},
\end{align*}
are nonempty and bounded. More importantly, there exists an optimal primal-dual solution pair with zero duality gap. That is, for optimal solutions
$X^*$ and $(y^*, S^*)$, we have
\begin{equation}
    C \bullet X^* - b^{\top} y^*  = X^* \bullet  S^*  = 0,
\end{equation}
which implies $X^* S^* = S^* X^* = 0$ as $X^*$ and $S^*$ are symmetric
positive semidefinite matrices. In primal-dual IPMs, the
complementarity condition $XS = 0$ that arises from the KKT optimality conditions is perturbed to
\begin{equation}\label{e:CP}
   XS = \mu I,
\end{equation}
where $I$ is the $n \times n$ identity matrix, and $\mu > 0$ is the
\textit{central path parameter}, which is (monotonically) reduced to
zero over the course of the algorithm.

Under the IPC and linear
independence of the matrices $A_1, \dots, A_m$, the central path equation
system
\begin{equation}\label{e:CP2}
    \begin{aligned}
    A_i \bullet X &= b_i \quad \text{ for } i \in [m], &X \succ 0,   \\
    \sum_{i \in [m]} y_i A_i + S &= C,&S \succ 0, \\
    XS &= \mu I, 
\end{aligned}
\end{equation}
has a unique solution for all $\mu > 0$ \cite{nesterov1988general}. Accordingly, the central path of \eqref{e:SDO}-\eqref{e:SDD} is the set of solutions to \eqref{e:CP2} for all $\mu > 0$. IPMs approximately follow the central path as $\mu \to 0$ by iteratively applying Newton's method to system \eqref{e:CP2} in a certain neighborhood of the central path. The best iteration complexity results are obtained using the Frobenius, or narrow neighborhood of the central path:
$$\Ncal_F (\gamma) = \left\{ (X, y, S) \in \Pcal^{\circ} \times \Dcal^{\circ} : \left\| X^{1/2} S X^{1/2} - \mu I \right\|_F \leq \gamma \mu \right\},$$
where $\mu = \frac{X \bullet S}{n}$. 

A classical Feasible IPM, as outlined in Algorithm \ref{alg:IPM} must be initialized to a primal-dual strictly feasible pair $(X,y,S)\in \Pcal^{\circ}\times \Dcal^{\circ}$ for the SDO primal and dual problems \eqref{e:SDO} and \eqref{e:SDD}. 
In each iteration, we systematically reduce $\mu \to \sigma \mu$ by a factor $\sigma <1$, and solve the \textit{Newton linear system} in order to update the solutions for the following iterate:
$$X \gets X + \Delta X, \quad 
    S \gets S + \Delta S. $$
In order to apply Newton's method to \eqref{e:CP2}, one needs to linearize the complementarity condition. This yields the following Newton linear system:
\begin{equation}\label{e:KP}
\begin{aligned}
\Delta XS + X\Delta S  &= \mu I - XS \\
\Delta X &\in \text{null}(A_1, \dots, A_m) \\
\Delta S &\in \left( \text{null}(A_1, \dots, A_m) \right)^{\perp}.
\end{aligned}
\end{equation}
It is well-established in the IPM literature that \eqref{e:KP} does not admit a symmetric
matrix solution (more specifically, a symmetric $\Delta X$), see, e.g.,
\cite{alizadeh1998primal}. Additional precautions need therefore to be taken because both the
primal and dual solutions must be symmetric.
 
\subsubsection{Symmetrizing the Newton System}
Here we review well-studied techniques to guarantee symmetry of $\Delta S$ and $\Delta X$ \cite{alizadeh1998primal, nesterov1997self, nesterov1998primal}. 

Adopting the presentation of Zhang \cite{zhang1998extending}, we define the linear transformation $H_P (M)$ that symmetrizes a matrix $M$ for a given invertible matrix $P$ to be:
\begin{equation}
    H_P(M) = \frac{1}{2} \left[ P M P^{-1} + \itran{P} \tran{M} \tran{P} \right].
\end{equation}
The symmetric form of the central path equations is therefore expressed as $H_P (XS) = \mu I$, and from the linearity of $H_P(\cdot)$, we have the symmetrized Newton linear system
\begin{equation}\label{e:tranP}
\begin{aligned}
    H_P (\Delta X S+ X \Delta S ) &= \sigma \mu I - H_P (XS), \\
    \Delta X &\in \text{null}(A_1, \dots, A_m), \\
    \Delta S &\in \left( \text{null}(A_1, \dots, A_m) \right)^{\perp}.
\end{aligned}
\end{equation}
Choosing $P$ to induce specific instances of this system, leads to different search directions that exhibit differing properties.
The class of matrices that are valid choices for $P$ form what is commonly referred to as the Monteiro-Zhang (MZ) family of search directions \cite{monteiro1998polynomial}. 
Table \ref{t:P} from \cite{todd1998nesterov} summarizes commonly employed choices for $P$, and their associated properties.

\begin{table}[h]
    \caption{Properties of Scaling Matrices \cite{todd1998nesterov}}\label{t:P}
    \centering
    \renewcommand{\arraystretch}{1.25}
    \begin{tabular}{llccc}
    \toprule  
    ~& ~ & Primal-Dual & Scale & Directions  \\
    Direction & $P$ & symmetry & invariance & uniquely defined  \\ \hline 
     NT \cite{nesterov1997self, nesterov1998primal}                         & $W^{-1/2}$ & yes & yes & yes  \\
    HKM \cite{helmberg1996interior, kojima1997interior, monteiro1997primal} & $S^{1/2}$  & no  & yes & yes  \\
    AHO \cite{alizadeh1998primal}                                           & $I$        & yes & no  & no   \\ 
    \botrule
    \end{tabular} 
\end{table}
Note that the Nesterov-Todd \cite{nesterov1998primal} scaling matrix $W$ is defined as:
\begin{equation}\label{e:NTdir}
    \begin{aligned}
    W &= S^{-1/2} (S^{1/2} X S^{1/2})^{1/2} S^{-1/2} \\
&= X^{1/2} (X^{1/2} S X^{1/2})^{-1/2} X^{1/2}.
    \end{aligned}
\end{equation}
It is well established that the Nesterov-Todd direction provides the strongest results. 

\begin{algorithm}[H]
\SetAlgoLined
\KwIn{Choose constants $\gamma$ and $\delta$ in $(0,1)$ and $\sigma = 1 - \frac{\delta}{\sqrt{n}}$.\\ 
\qquad      Choose $\left(X^{(0)}, y^{(0)}, S^{(0)}\right) \in \Ncal_F (\gamma)$.}
    \While{$n\mu^{(k)} > \epsilon$}{ 
    \begin{enumerate}
        \item Choose a nonsingular matrix $P^{(k)} \in \R{n \times n}$.
        \item Compute the solution to the Newton system \eqref{e:tranP} as $(\Delta X^{(k)}, \Delta y^{(k)}, \Delta S^{(k)})$.
        \item Choose a proper steplength $\alpha^{(k)} \in (0,1)$.
        \begin{align*}
           X^{(k+1)} &\leftarrow X^{(k)} + \alpha^{(k)} \Delta X^{(k)},   \\
           S^{(k+1)} & \leftarrow S^{(k)} + \alpha^{(k)} \Delta S^{(k)},~\quad y^{(k+1)} \leftarrow y^{(k)} +  \alpha^{(k)}\Delta y^{(k)}, \\
           \mu^{(k+1)} &\leftarrow  \frac{X^{(k+1)} \bullet S^{(k+1)}}{n}.
        \end{align*}
    \end{enumerate}
  } 
 \caption{Classical IPM}
\label{alg:IPM}
\end{algorithm}

\section{Iterative Refinement for SDO}\label{sec: IR}
In the classical IPM literature for solving LO problems, several studies address the numerical analysis of reaching an \textit{exact} solution, such as the Iterative Refinement scheme by \cite{Gleixner2016_Iterative} and the Rational Reconstruction method by \cite{Gleixner2020_Linear}.
In this section, we propose an Iterative Refinement method to solve SDOPs with high precision using limited-precision oracles. Here, we design a variant of IR for limited-precision oracles that provides interior-feasible solutions with a constant optimality gap $\epsilon$. The proposed IR method, presented as Algorithm \ref{alg:IR-SDO}, is adaptable to feasible IPMs, and we use this method in combination with the IF-(Q)IPM outlined in Section~\ref{sec: qipm}. The analysis of the IR scheme using the classical IF-IPM is provided in Appendix~\ref{sec: CGM}. In Appendix~\ref{sec: otheIR}, we extend this procedure to oracles producing solutions where both optimality gap and infeasibility are bounded by $\epsilon$. We point out that this work also serves as the first effort to develop iterative refinement for general SDO problems.  

The algorithm takes as input the data defining the primal-dual pair \eqref{e:SDO}-\eqref{e:SDD}, and two error tolerances: $\epsilon$, the \textit{fixed} precision to which each oracle call is made; and $\tilde{\epsilon}$, the desired duality gap of the final solution. At each iterate, we make a constant-precision call to the limited-precision oracle with the refining SDO problem data, which in turn reports the refining solution $(\overbar{X}, \bar{y}, \overbar{S})$. From here, the solution to the SDO problem is updated, and we update the duality gap associated with the current solution $X\bullet S$. If $X\bullet S$ has been reduced to, or below $\tilde{\epsilon}$, the algorithm terminates and reports $(X,y,S)$ as an $\tilde{\epsilon}$-optimal solution to \eqref{e:SDO}-\eqref{e:SDD}. Otherwise, we prepare a refining problem for our current solution and proceed to the next iteration.

\begin{algorithm}[h] 
\SetAlgoLined
\KwIn{Problem data $A_1, \dots, A_m, C \in \Scal^n$, $b \in \R{m}$, \\
  \qquad      Error tolerances $0 < \tilde{\epsilon} \ll \epsilon < 1$. \\
  \qquad     Strictly feasible point for \eqref{e:SDO}-\eqref{e:SDD} $(X^{(0)}, y^{(0)}, S^{(0)})$} 
\KwOut{An $\tilde{\epsilon}$-optimal solution pair  $(X, y, S)$ 
to the SDO problem $(A_1, \dots, A_m, b, C)$}
\textbf{Initialize}: $k \gets 1$ \\  
   $(X^{(1)}, y^{(1)}, S^{(1)}) \gets$ \textbf{solve} \eqref{e:SDO}-\eqref{e:SDD} to precision $\epsilon$  \\
   Compute scaling factor  
        $\displaystyle{\eta^{(1)} \gets\frac{1}{X^{(1)} \bullet S^{(1)}}.}$
\\
\While{$X^{(k)}\bullet S^{(k)} > \tilde{\epsilon}$}{
\begin{enumerate}
    \item $(\overbar{X}, \overbar{y}, \overbar{S}) \gets$ \textbf{solve} \eqref{e:SDO_ref}-\eqref{e:SDOD_ref} with input $(A_1, \dots, A_m, 0, X^{(k)}, \eta^{(k)}S^{(k)})$ \\
       to precision $\epsilon.$
    \item Update solution \\
    $\displaystyle{  X^{(k+1)} \gets X^{(k)} + \frac{1}{\eta^{(k)}} \overbar{X},}$ \\
    $\displaystyle{ y^{(k+1)} \gets y^{(k)} + \frac{1}{\eta^{(k)}} \overbar{y}, \quad  S^{(k+1)}\gets C-\sum_{i\in[m]} y_{i}^{(k+1)} A_i.}$\label{irstep}
    \item Update scaling factor $\displaystyle{\eta^{(k+1)} \gets \frac{1}{X^{(k+1)}\bullet S^{(k+1)} }}.$
    \item $k \gets k+1$
\end{enumerate}} 
\caption{Iterative Refinement for SDO}
\label{alg:IR-SDO}
\end{algorithm}

The primal and dual refining problems are presented in Definition \ref{theorem:iterative refinement idea F}.
\begin{definition} \label{theorem:iterative refinement idea F}
Consider the primal-dual SDO pair \eqref{e:SDO}-\eqref{e:SDD} and assume that the IPC is satisfied. Let $(X,y,S)$ be the current solution to \eqref{e:SDO}-\eqref{e:SDD}, and let
$\eta \geq 1$ be the scaling factor. 
The primal \textit{refining} problem is defined as
\begin{equation} \tag{$\bar{P}$} \label{e:SDO_ref}
\min_{\overbar{X}\in \Scal^n} \left\{   \eta S\bullet\overbar{X} :
    A_i \bullet \overbar{X} = 0 \text{ for }i \in [m] \text{ and } \overbar{X} \succeq -\eta X\right\},
\end{equation} 
and its dual problem  
\begin{equation}\tag{$\bar{D}$}\label{e:SDOD_ref}
\max_{(\bar{y}, \overbar{S}) \in \R{m} \times \Scal^n} \left\{  -\eta X \bullet\overbar{S} :
     \sum_{i \in [m]} \bar{y}_i A_i + \overbar{S}=\eta S \text{ and } \overbar{S}\succeq 0\right\}.
\end{equation} 
 
\end{definition}

The following result establishes that the sequence of iterates generated by Algorithm \ref{alg:IR-SDO} are increasingly accurate solutions to the primal-dual pair \eqref{e:SDO}-\eqref{e:SDD}.
\begin{theorem}\label{t:IR_SDO}
Let $\left(X^{(k)},y^{(k)}, S^{(k)}\right)$ be the current (overall) solution, and let $\left(\overbar{X},\bar{y},\bar{S}\right)$ be an $\epsilon$-optimal solution to the refining problem solved at iteration $k +1$ of Algorithm~\ref{alg:IR-SDO}. Then, 
$$\left(X^{(k+1)},y^{(k+1)},S^{(k+1)}\right) = \left(X^{(k)} + \frac{1}{\eta^{(k)}} \overbar{X},y^{(k)} + \frac{1}{\eta^{(k)}} \overbar{y}, C-\sum_{i\in[m]} y_{i}^{(k+1)} A_i\right)$$
is a strictly feasible solution for \eqref{e:SDO}-\eqref{e:SDD}.
\end{theorem}

\begin{proof}
The limited-precision oracle provides a solution $\left(\overbar{X}^{(k)},\bar{y}^{(k)}, \overbar{S}^{(k)}\right)$ to \eqref{e:SDO_ref}-\eqref{e:SDOD_ref} that satisfies
\begin{align*}
    A_i \bullet \overbar{X}^{(k)} &= 0, \;{\rm for } \;\; i \in [m] ,\\
    X^{(k)} + \frac{1}{\eta^{(k)}} \overbar{X}^{(k)} &\succeq 0,\\ 
    \overbar{S}^{(k)}=\eta^{(k)} S^{(k)} - \sum_{i\in [m]} \bar{y}_{i}^{(k)} A_i  &\succeq 0, \\
     \left(\eta^{(k)} X^{(k)}  +  \overbar{X}^{(k)} \right) \bullet \overbar{S}^{(k)} &\leq \epsilon.
\end{align*} 
Observe that $X^{(k+1)} = X^{(k)} + \frac{1}{\eta^{(k)}} \overbar{X}^{(k)} \succeq 0$. 

Next, note that for any $k \geq 1$, the updated solution will satisfy primal and dual feasibility. For all $i \in [m]$, we have
\begin{align*}
    A_i \bullet X^{(k+1)} = A_i \bullet \left(X^{(k)} + \frac{1}{\eta^{(k)}} \overbar{X}^{(k)} \right) &=   A_i \bullet X^{(k)} + 0 = b_i.
\end{align*}
Similarly, 
\begin{align*}
     S^{(k+1)}&=C - \sum_{i \in [m]}  y_i^{(k+1)} A_i\\ &= C - \sum_{i \in [m]}   \left(y_i^{(k)} + \frac{1}{\eta^{(k)}} \bar{y}_i^{(k)} \right) A_i \\
     &= \frac{1}{\eta^{(k)}}  \left( \eta^{(k)} \left(C - \sum_{i \in [m]}  y_i^{(k)} A_i \right) - \sum_{i \in [m]} \bar{y}^{(k)} A_i \right) \\
     &= \frac{1}{\eta^{(k)}}  \left( \eta^{(k)} S^{(k)} - \sum_{i \in [m]}\bar{y}_i^{(k)} 
     A_i \right) \succeq 0.
\end{align*}
The proof is complete.
\end{proof}

In the next section, we show that the IR method has quadratic convergence to the optimal solution set of the SDO problem.

\section{Quadratic Convergence of Iterative Refinement}\label{sec: convergence}
The IR method of Algorithm~\ref{alg:IR-SDO} reduces the optimality gap quadratically. The next result formalizes this fact in the setting where one has access to limited-precision oracles that output a feasible interior solution with $\epsilon$ duality gap.

\begin{theorem}\label{theo: Quadcon} Let $(X^{(0)},y^{(0)},S^{(0)})$ be a strictly feasible solution for the primal-dual SDO pair \eqref{e:SDO}-\eqref{e:SDD} and $\epsilon<1$. At each iteration of Algorithm~\ref{alg:IR-SDO}, we  have
$$X^{(k+1)}\bullet S^{(k+1)}\leq \epsilon\left(X^{(k)} \bullet S^{(k)}\right)^2.$$
\end{theorem}
\begin{proof}
As $\eta^{(k)}=\frac{1}{X^{(k)}\bullet S^{(k)}}$, we have
    \begin{align*}
     X^{(k+1)}\bullet S^{(k+1)} &=   \left(X^{(k)} +\frac{1}{\eta^{(k)}}\overbar{X}^{(k)} \right) \bullet \left(C- \sum_{i \in [m]} \left(y^{(k)}_i+ \frac{1}{\eta^{(k)}} \bar{y}^{(k)} \right) A_i \right)    \\
    &=\frac{1}{(\eta^{(k)})^2} \left[ \left(\eta^{(k)} X^{(k)} + \overbar{X}^{(k)} \right) \bullet \left(\eta^{(k)} S^{(k)} - \sum_{i \in [m]} \bar{y}^{(k)}_i A_i \right) \right]\\
    &=\left( X^{(k)} \bullet S^{(k)}\right)^2 \left[ \left(\eta^{(k)} X^{(k)} + \overbar{X}^{(k)} \right) \bullet \Sbar^{(k)}\right]\\
    &\leq\epsilon \left( X^{(k)} \bullet S^{(k)}\right)^2.
\end{align*} 
The proof is complete.
\end{proof}
\begin{corollary}\label{corr:irSDO}
Algorithm \ref{alg:IR-SDO} obtains an $\tilde{\epsilon}$-optimal solution to the primal-dual SDO pair \eqref{e:SDO}-\eqref{e:SDD} in at most 
$$ \Ocal \left(\log\log \left( \frac{1}{\tilde{\epsilon}} \right) \right)$$
iterations. 
\end{corollary}
\begin{proof}
The result is a straightforward consequence of Theorem \ref{theo: Quadcon}. Since $X^{(1)} \bullet S^{(1)}\leq\epsilon$,
Algorithm~\ref{alg:IR-SDO} stops when
$$X^{(k)} \bullet S^{(k)}\leq \epsilon^{2^{k}-1}\leq\tilde{\epsilon}.$$
Thus, we have 
$k\geq\frac{\log\log(\tilde{\epsilon})}{\log\log(\epsilon)}$. As the oracle has fixed-precision $\epsilon$, the iteration complexity of Algorithm~\ref{alg:IR-SDO} is $ \Ocal \left(\log\log \left( \frac{1}{\tilde{\epsilon}} \right) \right)$.
\end{proof}
The result in Theorem~\ref{theo: Quadcon} only relies on the assumption that the IPC is satisfied.
We emphasize that this assumption can be made without loss of generality: one always can embed the original SDO problem in a self-dual model that admits a trivial interior solution by construction \cite{de2006aspects}. 
Accordingly, the proposed IR method achieves quadratic convergence even in cases where strict complementarity fails, or the problem is degenerate. 
Another important point is the cost per iteration of the proposed IR as the refining SDO problem needs to be solved inexactly with fixed low precision.
In addition, in Step~\ref{irstep}, the matrix summations require $\Ocal(mn^2)$ arithmetic operations, which are dominated by the cost of solving the SDO refining problem.

The IR method of Algorithm~\ref{alg:IR-SDO} requires an SDO solver that delivers an interior feasible solution $(X,y,S)$ with optimality gap not exceeding a specified constant, and so an appropriate choice for our SDO oracle is a feasible IPM.
To analyze the per-iteration cost of Algorithm~\ref{alg:IR-SDO}, here we consider the setting in which each of the refining problems is solved using a short-step feasible IPM. 
Note that this introduces both practical and theoretical challenges to designing and analyzing our IR framework.

The use of a feasible IPM subroutine requires us to have access to an (easy to prepare) initial feasible interior solution $(\mathring{X}^{(k)},\mathring{y}^{(k)}, \mathring{S}^{(k)})$ for each the refining problems encountered during the run of the IR method.
At iteration $k$ of the IR method, our SDO subroutine, a feasible IPM, exhibits an iteration complexity of $\Ocal(\sqrt{n}\log(\frac{n\mathring{\mu}^{(k)}}{\epsilon}))$, where $\mathring{\mu}^{(k)}=\frac{\mathring{X}^{(k)}\bullet \mathring{S}^{(k)}}{n}$.
The dimension $n$ is given, and $\epsilon < 1$ is a given constant, but we still need to determine $\mathring{\mu}^{(k)}$.
Thus, we need to find an appropriate initial interior solution for the refining problem. 
One may think of embedding the refining problem in the self-dual embedding formulation in each step of IR and update the solution for the original problem accordingly. 
In this case, the IR method may have undetermined complexity, because the embedding problem does not provide a direct bound on the optimality gap for a solution that is retrieved from an $\epsilon$-optimal low precision solution of the embedding formulation. 
In Section~\ref{sec: embedding}, we discuss that the original SDO problem should be embedded once initially and the whole IR approach should be used to get a precise solution for the embedding formulation. 
Thus, we need to find an initial interior solution for the refining problem at each iteration of IR, even if the embedding formulation is used for the original problem.

Here, we present two appropriate choices for the initial interior solution of the refining problem.
It is worth mentioning that the refining SDO pair \eqref{e:SDO_ref}-\eqref{e:SDOD_ref} are not in standard form. 
By changing variable $(\hat{X},\hat{y}, \hat{S})=(\overbar{X}+\eta X,\bar{y}, \overbar{S})$, we reformulate them as
\begin{equation} \tag{$\hat{P}$} \label{e:standP_ref}
\min \left\{   \eta S\bullet \hat{X}- \eta^2 S\bullet X:
    A_i\bullet \hat{X} = \eta b \text{ for }i \in [m] \text{ and } \hat{X} \succeq 0\right\},
\end{equation} 
\begin{equation}\tag{$\hat{D}$}\label{e:standD_ref}
\max \left\{  \eta b^{\top} \hat{y} - \eta^{2} S\bullet X :
    \sum_{i \in [m]} \hat{y}_i A_i +\hat{S}= \eta S \text{ and }\hat{S}\succeq 0\right\}.
\end{equation} 
As we can see, the complementarity gap for this standard formulation is the same as the complementarity gap of the refining problem as $$\Xhat\bullet\Shat= (\overbar{X}+\eta X)\bullet\overbar{S}.$$ 
Let $(\mathring{X},\mathring{y},\mathring{S})$ be an interior solution for the original SDO problem and $(X^{(k)},y^{(k)},S^{(k)})$ be the solution at iteration $k$ of the IR method.
Then, one can choose the following two interior solutions for refining SDO pairs \eqref{e:standP_ref}-\eqref{e:standD_ref}:
\begin{enumerate}
    \item $(\eta^{(k)} \mathring{X}, \eta^{(k)}(y^{(k)}- \mathring{y}), \eta^{(k)} \mathring{S})$,
    \item $(\eta^{(k)} X^{(k)}, 0, \eta^{(k)} S^{(k)})$.
\end{enumerate}
By generalizing the result of \cite[Theorem 6.4]{mohammadisiahroudi2023inexact} from LO to SDO, one can easily show that these solutions are in the interior of the feasible region of the refining problem.
We have
$\mathring{\mu}^{(k)}=(\eta^{(k)})^2\mathring{\mu}$, and $\mathring{\mu}^{(k)}=(\eta^{(k)})^2\mu^{(k)}$ for the first and the second choices, respectively. In both choices, the $\mathring{\mu}^{(k)}$ is growing by the IR iterations, since $\eta^{(k)}$ is growing. 
Thus, although the IR method has quadratic convergence toward an optimal solution, the cost per iteration may grow. 
In the worst case, we have $n\mathring{\mu}^{(k)}=\Ocal(\frac{1}{\Tilde{\epsilon}})$ but it is not endangering polynomial complexity of the whole approach as the total complexity has a logarithmic dependence on $\mathring{\mu}^{(k)}$.
In appendix~\ref{sec: otheIR}, we propose two other variants of IR that are adaptable to infeasible oracles. 
Those IR methods converge to an optimal solution linearly.
Although the optimality gap decreases quadratically, infeasibility decreases linearly. 

To sum up, the iteration complexity of IR method is $\Ocal(\log\log(\frac{1}{\tilde{\epsilon}}))$. 
However, if we use an exact feasible IPM with the proposed initial interior solution, the cost per iteration is $\Ocal(n^{6.5}\log(\frac{n}{\tilde{\epsilon}}))$ due to initial solution with large complementarity gap. 
Although using exact feasible IPM in the proposed IR scheme is not a smart choice, it is a good example to show that the total complexity still may depend on target precision logarithmically.
The question for future research is if we can achieve total complexity with double logarithmic dependence on target precision. 
There are many unexplored directions to answer this question, such as finding the better initial interior solution for the refining problems, or a procedure to find such a solution. 
In another direction, one may choose or develop another inexact SDO solver with a fixed number of iterations. 
In the next section, we investigate how IR can be useful for inexact IPMs by combining the IR methodology with an IF-QIPM.

\section{Application to Quantum Interior Point Methods}\label{sec: qipm}

In this section, we give an overview of QIPMs for semidefinite optimization.
We begin by reviewing existing algorithms from the QIPM literature and then present a new algorithm that closely quantizes an IPM based on the Homogeneous Self Dual Embedding model. 

\subsection{Inexact-Feasible QIPM}

Let $\Acal^{\top} = [\textbf{\textup{vec}}(A_1), \textbf{\textup{vec}}(A_2), \cdots, \textbf{\textup{vec}}(A_m)]$. At each iteration of an IPM (classical or quantum) \cite{augustino2023quantum}, we seek to solve a system of the form 
\begin{equation}\label{e:newtonMotivate}
         \begin{pmatrix}
         0 & \tran{\cal A} & I \\
         {\cal A} & 0 & 0 \\
         \Ecal & 0 & \Fcal
         \end{pmatrix}
         \begin{pmatrix}
         {\rm \textbf{vec}} (\Delta X) \\
          \Delta y \\
         {\rm \textbf{vec}} (\Delta S)
         \end{pmatrix}
    = 
            \begin{pmatrix}
    0 \\ 0\\ R^c
    \end{pmatrix},
\end{equation}
where $R^c = \sigma \mu I - H_P (XS)$, $\Ecal = P \otimes SP^{-1} + P^{-1} S \otimes P$ and $\Fcal = PX \otimes P^{-1} + P^{-1} \otimes XP$. 
Solving the above system directly on a quantum computer, we obtain the quantum state proportional to the Newton step $\ket{\Delta X \circ \Delta y \circ \Delta S}$. Yet, any estimate of the corresponding classical estimate $(\Delta X, \Delta y, \Delta S)$ of the quantum state $\ket{\Delta X \circ \Delta y \circ \Delta S}$ obtained via quantum state tomography will only satisfy 
\begin{equation*} 
         \begin{pmatrix}
         0 & \tran{\cal A} & I \\
         {\cal A} & 0 & 0 \\
         \Ecal & 0 & \Fcal
         \end{pmatrix}
         \begin{pmatrix}
         {\rm \textbf{vec}} (\Delta X) \\
          \Delta y \\
         {\rm \textbf{vec}} (\Delta S)
         \end{pmatrix}
    = 
            \begin{pmatrix}
    \xi_d \\ \xi_p\\ R^c + \xi_c
    \end{pmatrix},
\end{equation*}
where $(\xi_d, \xi_p, \xi_c)$ are the errors to which the estimated solution satisfies primal feasibility, dual feasibility, and the complementarity condition, respectively. Consequently, it is not guaranteed that the primal and dual search directions $\Delta X$ and $\Delta S$ are members of orthogonal subspaces, which is required by the KKT optimality conditions. This in particular is one of the reasons the early works on QIPMs \cite{casares2020quantum, kerenidis2020quantum, kerenidis2021quantum} are not valid: their analysis relies on the assumption that $\Delta X \bullet \Delta S = 0$, which does not hold in this case.  

As we discussed earlier, the running time of II-QIPM does not indicate any speedup over classical (feasible) short-step IPMs for any parameter. As a result, the authors in \cite{augustino2023quantum} also devised an algorithm that would allow them to recover a feasible IPM framework. Indeed, the KKT optimality conditions stipulate that the primal and dual search directions are members of orthogonal subspaces. In particular, $\Delta X$ is an element of the null space of $\Acal$, denoted $\Ncal (\Acal)$, whereas $\Delta S$ is an element of the row space of $\Acal$, which we denote by $\Rcal (\Acal)$ in the sequel. 

Letting $\Acal^{\top}_s = [\textbf{\textup{svec}}(A^{(1)}), \textbf{\textup{svec}}(A^{(2)}), \cdots, \textbf{\textup{svec}}(A^{(m)})]$, it follows $\svec (\Delta X) \in \Ncal (\Acal_s)$ and $\svec (\Delta S) \in \Rcal (\Acal_s)$. Following \cite{augustino2023quantum}, one can define a basis for $\Rcal(\Acal_s)$ simply by choosing a basis for the null space via Gauss elimination or a QR-factorization of $\Acal$: 
$$ \Acal_s^{\top} = \begin{bmatrix} Q_1 & Q_2 \end{bmatrix} \begin{bmatrix} R \\ 0 \end{bmatrix}, $$
where $Q_1 \in \R{\frac{n(n+1)}{2} \times m}$, $Q_2 \in \R{\frac{n(n+1)}{2} \times \left(\frac{n(n+1)}{2} - m \right) }$, and $R \in \R{ m\times m}$. Moreover, the columns of $Q_2$ form a basis for the null space of $\Acal_s.$ Introducing a new variable $\lambda \in \R{ \left(\frac{n(n+1)}{2} - m \right)}$, the Newton directions $\svec( \Delta X) $ and $\svec(\Delta S)$ can be written as 
$$
    \textbf{svec}(\Delta X) =  Q_2 \lambda,\quad 
    \textbf{svec}(\Delta S) = - \Acal_s^{\top} \Delta y. 
$$

Let $P$ be an appropriate scaling matrix from the Monteiro-Zhang family that guarantees primal-dual symmetry, and define 
 \begin{align*}
    \Ecal_s &= (P \otimes_s P^{- \top} S),\\
    \Fcal_s &= (P X \otimes_s P^{-\top}), \\
     R^c  &= \sigma \mu I - H_P (XS),
 \end{align*}
where $\otimes_s$ denotes the symmetric Kronecker product. The Newton linear system for the IF-QIPM is given by
\begin{equation}\label{e:IF-Newton}\tag{OSS}
    \begin{bmatrix} 
    \Ecal_s Q_2 & \Fcal_s (-\Acal_s^{\top}) 
    \end{bmatrix}
        \begin{bmatrix} \lambda \\ \Delta y
    \end{bmatrix} =
    \svec (R^c).
\end{equation}
Let us denote the error introduced by QLSA by $r^c \in \Scal^{n \times n}$. We require that the error is proportional to the central path parameter, that is:
\begin{equation}\label{e:RcAssumption1}
    \| r^c \|_F \leq \beta\mu, \tag{AR1}
\end{equation}
for some $\beta \in (0,1)$.

\begin{theorem}[Propositions 10-12 in \cite{augustino2023quantum}]
Given $(X,y,S)\in \Pcal^{\circ}\times \Dcal^{\circ}$, we have
\begin{enumerate}
    \item System (\ref{e:IF-Newton}) is equivalent to System \eqref{e:KP}.
    \item System (\ref{e:IF-Newton}) has a unique solution $(\lambda, \Delta y)$.
    \item For  the solution $(\Delta X, \Delta y, \Delta S)$, where $\Delta X = {\rm \bf smat}(Q_2 \lambda)$ and $\Delta S = - \sum_{i\in[m]}\Delta y_i A_i $, one has $\Delta X \bullet \Delta S=0$.
    \item For an inexact solution $(\overline{\Delta X}, \overline{\Delta y}, \overline{\Delta S})$ with residual $r^c$, feasibility holds: 
    $$A_i \bullet (X+\alpha\overline{\Delta X})=b_i, \text{ and }\sum (y_i+\alpha\overline{\Delta y_i})A_i +(S+\alpha \overline{\Delta S})=C.$$
\end{enumerate}
\end{theorem}

Though we treat the IF-(Q)IPMs as a black-box SDO solver in order to generalize our results to any appropriate primal-dual algorithm, we briefly summarize how the quantum algorithm is implemented using a QLSA. Aside from how the Newton system is solved, the steps of a QIPM are exactly the same as the ones in a classic IPM. Thus, suppose we are at iterate $k$, with the current solution $\left( X^{(k)}, y^{(k)}, S^{(k)}\right)$. In order to prepare and solve the Newton linear system for the Newton step $\left(\Delta X, \Delta y, \Delta S\right)$ at iterate $k+1$, we  \textit{classically compute} the scaling matrix $P$, and store $P$ and $\left( X^{(k)}, y^{(k)}, S^{(k)}\right)$ in \textit{quantum random access memory} (QRAM). QRAM is a quantum analogue to the standard RAM, and a description can be found in \cite[Section 2.2]{chakraborty2018power}. From here, the data we have stored in QRAM is used to prepare a unitary \textit{block-encoding} \cite{chakraborty2018power} of the coefficient matrix of the Newton system, which is a technique for preparing quantum gates involving non-unitary matrices. Using our block-encoding of the Newton system coefficient matrix, and a quantum state encoding the right-hand side vector, we solve the quantum linear system to prepare a quantum state encoding the Newton step. A classical estimate of this step is produced up to precision $\epsilon$ through the use of \textit{quantum state tomography} \cite{van2022quantum}. From here, we use our classical estimate of the Newton step to update the solution, and we proceed to the next iterate. For a more detailed discussion, see \cite[Section 2]{augustino2023quantum}.
\begin{algorithm}[H]
\SetAlgoLined
\KwIn{Problem data $A_1, \dots, A_m, C \in \Scal^n$, $b \in \R{m}$, \\
  \qquad     Strictly feasible point for \eqref{e:SDO}-\eqref{e:SDD} $(X^{(0)}, y^{(0)}, S^{(0)})$\\
  \qquad Choose constants $\beta$, $\gamma$ and $\delta$ in $(0,1)$ and $\sigma = 1 - \frac{\delta}{\sqrt{n}}$.} 
\KwOut{An $\epsilon$-optimal solution pair  $(X, y, S)$ 
to the SDO problem $(A_1, \dots, A_m, b, C)$}
\textbf{Initialize}: $k \gets 0$ \\  
   Calculate $Q_2$  by QR factorization of $\mathcal{A}$\\
   $\mu^0\gets \frac{X^0\bullet S^0}{n}$\\
    \While{$n\mu^{(k)} > \epsilon$}{ 
    \begin{enumerate}
        \item ($\lambda^{(k)},\Delta y^{(k)}) \gets$ \textbf{solve} system (\ref{e:IF-Newton}) using QLSA with error bound $\epsilon^{(k)}$ 
        \item $ \Delta X^{(k)} \gets {\rm \bf smat}(Q_2 \lambda^{(k)})$ and $\Delta S^{(k)} = - \sum_{i\in[m]}\Delta y_i^{(k)} A_i$ 
        \item $(X^{(k+1)},y^{(k+1)},S^{(k+1)})\gets (X^{(k)},y^{(k)},S^{(k)})+(\Delta X^{(k)},\Delta y^{(k)},\Delta S^{(k)})$
        \item $\mu^{(k+1)}\gets\frac{(X^{(k)})^{\top}S^{(k)}}{n}$
        \item $k\gets k+1$
    \end{enumerate}
  } 
 \caption{IF-QIPM}
\label{alg: IF-IPM}
\end{algorithm}

The following two results from \cite{augustino2023quantum} bound the iteration complexity and overall running time of the IF-QIPM for SDO, respectively. 

\begin{theorem}[Theorem 4 in \cite{augustino2023quantum}]\label{t:Converge}
Let $\gamma, \beta \in (0,1)$ and $\delta \in (0,1)$ be constants satisfying 
\begin{equation*}
   \frac{2 \sqrt{2} \gamma}{1 - \gamma} \leq 1,~~ \beta \sigma  \leq \sqrt{ \frac{\gamma^2 + (1-\sigma)^2 n}{1- \gamma}},~~\beta \leq  1 -   \frac{\gamma}{\sqrt{n}} - \frac{ 21.7 (\gamma^2 + \delta^2)}{ (2 + \sqrt{2}) \left( 1 - \frac{\delta}{\sqrt{n}}\right) \gamma (1- \gamma)}.
\end{equation*}
Suppose that $(X, y, S) \in \Ncal_F (\gamma)$ and let $(\Delta X, \Delta y, \Delta S)$ denote the solution that we obtain from solving system \eqref{e:IF-Newton}, where 
$\sigma = 1 - \delta/\sqrt{n}$, and $\mu = (X \bullet S)/n$. Then,
\begin{itemize}
    \item[(a)] $(\widehat{X}, \hat{y}, \widehat{S}) = (X + \Delta X,  y + \Delta y, S + \Delta S) \in \Ncal_F (\gamma);$
    \item[(b)] $\widehat{X} \bullet \widehat{S} = \left( 1 - \frac{\delta}{\sqrt{n}} \right) (X \bullet S).$
\end{itemize}
\end{theorem}

\begin{corollary}[Corollary 2  in \cite{augustino2023quantum}]
\label{corr:runtimeIF}
A quantum implementation of the IF-QIPM with access to QRAM outputs an $\epsilon$-optimal solution $(X^*, y^*, S^*)$ to the primal-dual SDO pair \eqref{e:SDO}-\eqref{e:SDD} using at most
$$ \Ocal \left( n^{3.5} \frac{\kappa^2}{\epsilon} \cdot  \operatorname{polylog} \left(n, \kappa, \frac{1}{\epsilon} \right) \right)  $$
QRAM accesses and $ \Ocal \left( n^{4.5} \cdot  \operatorname{polylog} \left( \kappa, \frac{1}{\epsilon} \right) \right) $ arithmetic operations.
\end{corollary}

\subsection{Analyzing the Orthogonal Subspaces System}
In this section, we investigate the condition number of system \eqref{e:IF-Newton} because quantum solvers are sensitive to the condition number of the linear systems. 
Here, we extend the condition number analysis of OSS, done in \cite{mohammadisiahroudi2023inexact} for LO, to SDO. To compute the condition number of $M^{(k)}=\begin{bmatrix} 
    -\Fcal_s \Acal_s^{\top}& \Ecal_s Q_2
    \end{bmatrix}$, we compute the condition number of $(M^{(k)})^{\top} M^{(k)}$ which can be written as follows
\begin{align*}
    \left(M^{(k)}\right)^{\top} M^{(k)} 
    &= \begin{bmatrix}
    \Acal_s &0\\
    0 & Q_2^{\top}
    \end{bmatrix}\begin{bmatrix}
    (\Fcal_s)^2 & -(\Fcal_s)^{\top}\Ecal_s\\
    -(\Ecal_s)^{\top}\Fcal_s & (\Ecal_s)^2
    \end{bmatrix}\begin{bmatrix}
    \Acal^{\top} &0 \\
    0& Q_2
    \end{bmatrix}.
\end{align*}
Let $e_{\min}(X)$ and $e_{\max}(X)$ be the smallest and largest  eigenvalues of $X$, respectively. To analyze the OSS system, it is a prevailing assumption that 
$$\max\{e_{\max}(X),e_{\max}(S)\}\leq \omega.$$
In some SDOPs, $\omega$ can be exponentially large \cite{pataki2021exponential}. The Following Lemma provides a lower bound for the smallest eigenvalues of $X$ and $S$.
\begin{lemma}\label{lem: eigs}
    For $(X,y,S)\in \Ncal_{F}(\gamma)$, we have
    \begin{align}
        (1-\gamma)\frac{\mu}{\omega}&\leq e_{\min}(X),& e_{\max}(X)&\leq \omega,\\
        (1-\gamma)\frac{\mu}{\omega}&\leq e_{\min}(S),& e_{\max}(S)&\leq \omega.
    \end{align}
\end{lemma}

We recall the following results from \cite{mohammadisiahroudi2023inexact}.
\begin{lemma}
For any full row-rank matrix $T\in \Rmbb^{m\times n}$ and any symmetric positive definite matrix $D\in \Rmbb^{n\times n}$, their condition number satisfies
$$\kappa\left(TDT^{\top}\right)\leq \kappa(D)\kappa\left(TT^{\top} \right).$$
\end{lemma}
Let 
$$D^{(k)}=\begin{bmatrix}
    \left(\Fcal_s^{(k)}\right)^2 & -\left(\Fcal_s^{(k)}\right)^{\top}\Ecal_s^{(k)}\\
    -\left(\Ecal_s^{(k)}\right)^{\top}\Fcal_s^{(k)} & \left(\Ecal_s^{(k)}\right)^2
    \end{bmatrix}, \quad T=\begin{bmatrix}
    \Acal^{\top} &0 \\
    0& Q_2
    \end{bmatrix}.$$ 
Then an upper bound for the condition number of the \eqref{e:IF-Newton} system can be derived as
$$\kappa(M^{(k)})=\sqrt{\kappa\left(TDT^{\top}\right)}\leq \sqrt{\kappa(D^{(k)})}\kappa_T,$$
where $\kappa_T$ is the condition number of the matrix $T$ defined above. Based on Lemma~\ref{lem: eigs}, it is easy to verify that $\sqrt{\kappa(D^{(k)})}$ is dependent on $\omega$ and $\frac{1}{\mu}$. For the case of LO, authors of \cite{mohammadisiahroudi2023inexact} showed $\sqrt{\kappa(D^{(k)})}=\Ocal(\frac{\omega^2}{\mu})$. For general SDO, for different choices of symmetrization matrix $P$ the powers of $\omega$ and $\frac{1}{\mu}$ may have different exponents.
\begin{claim}\label{claim:con}
    There exist positive integers $p>0$ and $q>0$ such that the condition number of system \eqref{e:IF-Newton} is $\Ocal(\frac{\omega^p}{\mu^q}\kappa_T)$.
\end{claim}
Although it is not easy to prove that Claim~\ref{claim:con} holds for any symmetrization matrix $P$, it is   
straightforward to check that it is true for matrices presented in Table~\ref{t:P}. The reason is that $\Fcal_s^{(k)}$ and $\Ecal_s^{(k)}$ can be decomposed to following elements
\begin{align*}
    \kappa(X)=\Ocal \left(\frac{\omega^2}{\mu} \right),\kappa\left(X^{-1}\right)=\Ocal\left(\frac{\omega^2}{\mu}\right),\kappa\left(X^{\frac{1}{2}}\right)=\Ocal\left(\frac{\omega}{\sqrt{\mu}}\right),\kappa\left(X^{-\frac{1}{2}}\right)=\Ocal\left(\frac{\omega}{\sqrt{\mu}}\right),\\
    \kappa(S)=\Ocal\left(\frac{\omega^2}{\mu}\right),\kappa\left(S^{-1}\right)=\Ocal\left(\frac{\omega^2}{\mu}\right),\kappa\left(S^{\frac{1}{2}}\right)=\Ocal\left(\frac{\omega}{\sqrt{\mu}}\right),\kappa\left(S^{-\frac{1}{2}}\right)=\Ocal\left(\frac{\omega}{\sqrt{\mu}}\right).
\end{align*}
The following Theorem provides an upper bound for the condition number of system \eqref{e:IF-Newton} for the AHO direction.
\begin{theorem}\label{theo: conditionAHO}
    Let $P=I$. We have $\kappa(M^{(k)})=\Ocal\left(\frac{\omega^2}{\mu}\kappa_T\right)$.
\end{theorem}
\begin{proof}
    Based on Section 4 of \cite{alizadeh1998primal}, we have 
    $$\kappa(E^{(k)})=\Ocal\left(\frac{\omega^2}{\mu}\right)$$ where $E^{(k)}=\begin{bmatrix}
        \Fcal_s^{(k)} &\Ecal_s^{(k)}
    \end{bmatrix}$. Since $D^{(k)}=(E^{(k)})^{\top}E^{(k)}$, we have 
    $$\kappa(D^{(k)})=\Ocal\left(\frac{\omega^4}{\mu^2}\right),\quad \kappa(M^{(k)})=\Ocal\left(\frac{\omega^2}{\mu}\kappa_T\right).$$
    The proof is complete.
\end{proof}
As we can see, the condition number of the OSS system may grow to infinity as $\mu \to 0$. 
We show that in our IR method, we solve SDO problems with limited precision, and the general upper bound for the condition number is bounded by $\Ocal(\omega^2\kappa_T)$, which is constant and depends only on input data.

\subsection{An Inexact-Feasible QIPM for the Homogeneous Self-dual Embedding Model}\label{sec: embedding}
In this section, we use the canonical formulation for SDOPs as
\begin{align}\label{e:SDO canonical}
     z_P &= \inf_X \left\{ C \bullet X: A_i \bullet X +u_i= b_i,~\forall i \in [m], \ X \succeq 0, \ u\geq 0 \right\},\\  \label{e:SDO D-canonical}
    z_D &=  \sup_{y, S} \left\{ b^{\top} y:\sum_{i \in [m]} y_i A_i + S = C,~S\succeq 0, \ y \geq 0 \right\}.
\end{align}
One can easily change the standard formulation to canonical formulation and vice versa. 
In line with our discussion on the central path, recall that for a feasible IPM, we must assume the IPC, e.g., that a strictly feasible pair $X$ and $(y, S)$ with $X \succ 0$ and $S \succ 0$ exists \cite{de1997initialization}.
It is known that with the self-dual embedding model this condition may be assumed without loss of generality \cite{de1997initialization}. 
To see this, note that any primal-dual SDO pair of the form \eqref{e:SDO canonical} has a self-dual embedding formulation given by
\begin{align*}
    &\min \;\; (n+m+2)\theta \\
    &\begin{matrix}
    & &A_i \bullet X&-b_i\tau&+ \bar{b}_i\theta&-u_i&= 0,\hfill\\
    &-\sum_{j \in [m]} y_j A_j & &+C\tau&- \overbar{C}\theta&-S&= 0,\hfill\\
    &b^{\top} y & - C \bullet X &  &+ \bar{o}\theta &-\phi&= 0,\hfill\\
    &-\bar{b}^{\top} y&+ \overbar{C} \bullet X&-\bar{o} \tau& &-\rho &= -(n+m+2),\\
    \end{matrix}\\
    &X\succeq 0, \; S\succeq 0, \; \tau\geq 0, \; \theta\geq 0, \; \phi\geq 0, \; \rho\geq 0, \;u\geq 0, \; y\geq 0,
\end{align*}
where 
\begin{align*}
    &\bar{b_i}=b_i+1- A_i \bullet I,\\
    &\overbar{C}=C-I-\sum_{i \in [m]} A_j,\\
    &\bar{o}=1+ \overbar{C} \bullet I -b^{\top}e.
\end{align*}
Then, it can be easily verified that $y^0=u^0=e$, $X^0=S^0=I$ and $\theta^0=\tau^0=\phi^0=\rho^0=1$ is a feasible interior starting solution. 

Let $(y^*,X^*,\tau^*,\theta^*,u^*,S^*,\phi^*,\rho^*)$ denote a maximally complementary optimal solution of the self-dual embedding problem, then:
\begin{itemize}
    \item[(a)] if $\tau^*>0$, then a primal-dual complementary pair $(\frac{1}{\tau^*}X^*,\frac{1}{\tau^*}S^*)$ is obtained for the original primal and dual problems,
    \item[(b)] if $\tau^*=0$ and $\phi^*>0$, then a primal and/or dual improving ray is detected,
    \item[(c)] if $\tau^*=\phi^*=0$, then no complementary pair exists and neither primal nor dual improving ray exists.
\end{itemize}
The feasible Newton system for this formulation using the {\rm \bf svec} and $\Acal_s$ notation is as follows:
{\small
\begin{equation}\label{eq: newton system of self-dual}
\begin{aligned}
    \begin{matrix}
    & &\mathcal{A}_s\textbf{svec}(\Delta X^{(k)}) &-b\Delta \tau^{(k)}  &+\bar{b}\Delta\theta^{(k)}&-\Delta u^{(k)}&= 0,\\
    &-\mathcal{A}_s^{\top}\Delta y^{(k)} & &+\textbf{svec}(C)\Delta \tau^{(k)} &+ \textbf{svec}(\overbar{C})\Delta\theta^{(k)}&-\textbf{svec}(\Delta S^{(k)})&= 0,\\
    &b^{\top} \Delta y^{(k)}& -\textbf{svec}(C)^{\top} \textbf{svec}(\Delta X^{(k)}) & & + \bar{o}\Delta \theta^{(k)}&-\Delta \phi^{(k)}&= 0,\\
    &-\bar{b}^{\top} \Delta y^{(k)}&-\textbf{svec}(\overbar{C})^{\top} \textbf{svec}(\Delta X^{(k)})&-\bar{o} \Delta \tau^{(k)}  &&-\Delta\rho^{(k)}&= 0,\\
    \end{matrix}\\
    H_P(X^{(k)}\Delta S^{(k)}+ \Delta X^{(k)} S) =\sigma \mu^{(k)} I- H_P(X^{(k)} S^{(k)}) ,\qquad\qquad\qquad\qquad\\
    Y^{(k)}\Delta u^{(k)}+ U^{(k)}\Delta y^{(k)}=\sigma \mu^{(k)} e -Y^{(k)} u^{(k)},\qquad\qquad\qquad\qquad\\
    \tau^{(k)}\Delta\phi+\phi^{(k)}\Delta\tau^{(k)} = \sigma\mu^{(k)}- \tau^{(k)}\phi^{(k)},\qquad\qquad\qquad\qquad\\
    \theta^{(k)}\Delta\rho^{(k)}+\rho^{(k)}\Delta\theta^{(k)} =\sigma\mu^{(k)}-\theta^{(k)}\rho^{(k)},\qquad\qquad\qquad\qquad
\end{aligned}
\end{equation}
}
where $Y^{(k)}={\rm diag}(y^{(k)})$ and $U^{(k)}={\rm diag}(u^{(k)})$.

To derive the OSS system for Newton system \eqref{eq: newton system of self-dual}, we define 
\begin{equation}\label{eq: self-dual matrices}
\begin{aligned}
    \mathcal{P}&=\begin{bmatrix}
    I &\mathbf{0}&\mathbf{0} &\mathbf{0}&\mathbf{0} &-\mathcal{A}_s&b &-\bar{b}\\
    \mathbf{0}&I&\mathbf{0}&\mathbf{0}&\mathcal{A}_s^{\top}&\mathbf{0}&-\textbf{svec}(C)&-\textbf{svec}(\overbar{C})\\
    \mathbf{0}&\mathbf{0}& 1 &\mathbf{0}&-b^{\top} &\textbf{svec}(C)^{\top} &\mathbf{0}&-\bar{o}\\
    \mathbf{0}&\mathbf{0}&\mathbf{0}& 1 &\bar{b}^{\top}&\textbf{svec}(\overbar{C})^{\top} &\bar{o}^{\top} &\mathbf{0} \\
    \end{bmatrix},\\
    \mathcal{D}&=\begin{bmatrix}
    Y^{(k)}&\mathbf{0}&\mathbf{0}&\mathbf{0}&U^{(k)}&\mathbf{0}&\mathbf{0}&\mathbf{0}\\
    \mathbf{0}&\Fcal_s &\mathbf{0}&\mathbf{0}&\mathbf{0}&\Ecal_s& \mathbf{0}&\mathbf{0}\\
    \mathbf{0}&\mathbf{0}&\tau^{(k)}&\mathbf{0}&\mathbf{0}&\mathbf{0}&\phi^{(k)}&\mathbf{0}\\
    \mathbf{0}&\mathbf{0}&\mathbf{0}&\theta^{(k)}&\mathbf{0}&\mathbf{0}&\mathbf{0}&\rho^{(k)}
    \end{bmatrix},\\
    \mathcal{R}&=\begin{bmatrix}
    \sigma \mu^{(k)} e-Y^{(k)}u^{(k)}\\
    \sigma \mu^{(k)} e- H_P(X^{(k)}S^{(k)}) \\
    \sigma \mu^{(k)}-\tau^{(k)}\phi^{(k)}\\
    \sigma \mu^{(k)}-\theta^{(k)}\rho^{(k)}
    \end{bmatrix},\\
    \Delta \mathcal{X}&=(\Delta u^{(k)}, \textbf{svec}(\Delta S^{(k)}),\Delta \phi ^{(k)},\Delta\rho^{(k)}, \Delta y^{(k)},\textbf{svec}(\Delta X^{(k)}),\Delta \tau^{(k)},\Delta \theta^{(k)})^{\top},
\end{aligned}
\end{equation}
where $\mathbf{0}$ is the all-zero matrix. Then, the Newton system can be simplified as
\begin{equation}\label{eq: newton system of self-dual2}
\begin{aligned}
   \Delta \mathcal{X}&\in \text{Null}(\mathcal{P}), \\
   \mathcal{D}\Delta \mathcal{X}&=\mathcal{R}.
\end{aligned}
\end{equation}
A basis for the null space of $\mathcal{P}$ is given by the  column vectors of 
$$\mathcal{V}=\begin{bmatrix}
    \mathbf{0} &-\mathcal{A}&b &-\bar{b}\\
    \mathcal{A}^{\top}&\mathbf{0}&-\textbf{svec}(C)&-\textbf{svec}(\overbar{C})\\
    -b^{\top} &\textbf{svec}(C)^{\top} &\mathbf{0}&-\bar{o}\\
    \bar{b}^{\top}&\textbf{svec}(\overbar{C})^{\top} &\bar{o}^{\top} &\mathbf{0} \\
    -I &\mathbf{0}&\mathbf{0} &\mathbf{0}\\
    \mathbf{0}&-I&\mathbf{0}&\mathbf{0}\\
    \mathbf{0}&\mathbf{0}& -1 &\mathbf{0}\\
    \mathbf{0}&\mathbf{0}&\mathbf{0}& -1 
    \end{bmatrix}.$$
This demonstrated that if we have an SDOP in canonical form, then a basis for the null space of $\mathcal{P}$ can be constructed without any factorization or elimination procedure. 
Thus, the cost of preprocessing is negligible in this setting.
The OSS for this formulation at the $k^{\rm th}$  iteration is
\begin{equation}\label{eq: OSS of self-dual}
\mathcal{D}^{(k)}\mathcal{V}\lambda^{(k)}=\mathcal{R}^{(k)},
\end{equation}
where $\mathcal{\lambda}^{(k)}\in\mathbb{R}^{\frac{n(n+1)}{2}+m+2}$ and the size of the system will be $\frac{n(n+1)}{2}+m+2$. We can calculate the Newton direction by $\Delta \mathcal{X}=\mathcal{V}\mathcal{\lambda}^{(k)}$. 
Even if $\tilde{\lambda}^{(k)}$ is an inexact solution of system \eqref{eq: OSS of self-dual}, then the inexact direction $\widetilde{\Delta \mathcal{X}}^{(k)} = \Vcal \Tilde{\lambda}^{(k)}$ is still always a feasible direction, since $\widetilde{\Delta \mathcal{X}}^{(k)}\in \text{Null}(\mathcal{A})$.

To have a convergent IF-IPM, we need that for the error $\|r^{(k)}\|\leq \beta \mu^{(k)}$ holds, where $r^{(k)}=\mathcal{D}\mathcal{V}\tilde{\lambda} - \mathcal{R}
 = \mathcal{D}\mathcal{V}(\tilde{\lambda}^{(k)}-\lambda^{(k)})$. So, the error bound $\epsilon^{(k)}=\frac{\beta\mu^{(k)}}{\|\mathcal{D}\mathcal{V}\|}$ is needed.

\begin{lemma}\label{lemma: oss for self-dual}
Let $( u^{(k)},  S^{(k)}, \phi ^{(k)},\rho^{(k)},  y^{(k)}, X^{(k)}, \tau^{(k)}, \theta^{(k)})\in \mathcal{PD}_0$ then the following statements hold.
\begin{enumerate}
    \item Systems \eqref{eq: OSS of self-dual} and \eqref{eq: newton system of self-dual} are equivalent.
    \item System \eqref{eq: OSS of self-dual} has a unique solution.
    \item For the solution of system \eqref{eq: OSS of self-dual} orthogonality holds, i.e.,  \\
     $(\Delta X^{(k)}\bullet \Delta S^{(k)})+(\Delta y^{(k)})^{\top}\Delta u^{(k)}+\Delta \tau^{(k)} \Delta \phi^{(k)}+\Delta \theta \Delta \rho =0$.
\end{enumerate}

\end{lemma}

Now, we can develop our IF-QIPM for the self-dual embedding model as Algorithm~\ref{alg: IF-IPM for self-dual}.
\begin{algorithm}[H]
\SetAlgoLined
\KwIn{Choose constants $\beta$, $\gamma$ and $\delta$ in $(0,1)$ and $\sigma = 1 - \frac{\delta}{\sqrt{n}}$.\\ 
\quad $(y^0,X^0,\tau^0,\theta^0,u^0,S^0,\phi^0,\rho^0)\gets(e,I,1,1,e,I,1,1)$\\
\quad 
$\mu^0\gets 1$}
    \While{$(n+m+2)\mu^{(k)} > \epsilon$}{ 
    \begin{enumerate}
        \item $\epsilon^{(k)} \gets \beta\frac{\mu^{(k)}}{\|\mathcal{D}^{(k)}\mathcal{V}\|}$
        \item $\lambda^{(k)} \gets$ \textbf{solve} system (\ref{eq: OSS of self-dual}) with error bound $\epsilon^{(k)}$ 
        \item $\Delta \mathcal{X}^{(k)}= \mathcal{V}\lambda^{(k)}$  \label{alg-step:x and s-P}
        \item $\mathcal{X}^{(k+1)} \gets \mathcal{X}^{(k)}+\Delta \mathcal{X}^{(k)}$
        \item$\mu^{(k)}\gets\frac{(y^{(k)})^{\top}u^{(k)}+X^{(k)}\bullet S^{(k)})+\tau^{(k)} \phi ^{(k)}+\rho^{(k)} \theta^{(k)}}{n+m+2}$
    \end{enumerate}
  } 
 \caption{IF-QIPM for the Self-dual Embedding Formulation}
\label{alg: IF-IPM for self-dual}
\end{algorithm}
In a QIPM we solve system (\ref{eq: OSS of self-dual}) using a QLSA. The convergence analysis of the IF-QIPM of Algorithm~\ref{alg: IF-IPM for self-dual} applied to the self-dual embedding formulation is analogous to the analysis of the IF-QIPM of \cite{augustino2023quantum}.
%
\begin{theorem}\label{theorem: convergence of IF-QIPM for self-dual}
 For the IF-QIPM of Algorithm~\ref{alg: IF-IPM for self-dual}  applied to the self-dual embedding formulation, the following statements hold.
 \begin{enumerate}
     \item The sequence $\{\mu_k\}_{k\in \mathbb{N}}$ converges linearly to zero.
     \item For any $k\in \mathbb{N}$, $\mathcal{X}^{(k)}\in \mathcal{N}_F(\gamma)$.
     \item After at most $\Ocal(\sqrt{n}\log(\frac{n}{\epsilon}))$ iterations, the IF-QIPM delivers an $\epsilon$-optimal solution.
     \item The IF-QIPM finds an $\epsilon$-optimal solution  using at most
$$ \Ocal \left( n^{3.5} \frac{\omega^{4}\kappa_T^2}{\epsilon^{3}} \cdot  \operatorname{polylog} \left(n, \kappa, \frac{1}{\epsilon} \right) \right)  $$
QRAM accesses and $ \Ocal \left( n^{4.5} \cdot  \operatorname{polylog} \left( \kappa, \frac{1}{\epsilon} \right) \right) $ arithmetic operations.
 \end{enumerate}
\end{theorem}
\begin{proof}
The proof is straightforward by combining the results of Theorems~\ref{t:Converge}, \ref{theo: conditionAHO},  and Corollary~\ref{corr:runtimeIF} for the self-dual embedding formulation.
\end{proof}
%
As we can see, the IF-QIPM has polynomial dependence on $\frac{1}{\epsilon}$ which means that it can be used for finding an inexact solution for an SDOP. To find a precise solution and improve the complexity with respect to inverse precision, we use the IR method of Algorithm~\ref{alg:IR-SDO} calling IF-QIPM as a subroutine.  
We first embed the original problem in the self-dual formulation and then apply the proposed IR-IF-QIPM, IR method of Algorithm \ref{alg:IR-SDO}, where the refining problem is solved by the IF-QIPM of Algorithm \ref{alg: IF-IPM for self-dual}. In this framework, one can set the target precision of the IF-QIPM to $\epsilon=10^{-2}$, and the final precision of IR method as needed. 
The following theorem presents the total complexity of the proposed IR-IF-QIPM for solving an SDO problem 
\begin{theorem}\label{theo: IR-IF-QIPM}
    The proposed IR-IF-QIPM, the IR method of Algorithm \ref{alg:IR-SDO} augmented with the IF-QIPM of Algorithm \ref{alg: IF-IPM for self-dual}, finds an $\epsilon$-optimal solution using at most
$$ \Ocal \left( n^{3.5}\omega^{4}\kappa_T^2 \cdot  \operatorname{polylog} \left(n, \kappa, \frac{1}{\epsilon} \right) \right)  $$
QRAM accesses and $ \Ocal \left( n^{4.5} \cdot  \operatorname{polylog} \left( \kappa, \frac{1}{\epsilon} \right) \right) $ arithmetic operations.
\end{theorem}
The proof of Theorem~\ref{theo: IR-IF-QIPM} follows from Theorem~\ref{theo: convergence of iterative refinement} and Theorem~\ref{theorem: convergence of IF-QIPM for self-dual}. As we can see, the proposed IR-IF-QIPM shows exponential speed-up with respect to precision. In the next section, we compare our results with other types of SDO solvers. In addition, Appendix~\ref{sec: CGM} shows that the proposed IR can speed up the classical IF-IPM, that use classical iterative methods to solve the OSS system.

\subsection{Comparison to existing SDO 
         solvers}
Given the advantageous running times we obtain for our IR-IF-QIPMs, we should determine how our classical and quantum algorithms compare to the best-performing methods in both the classical and quantum literature. 

In Table \ref{tab:runTnsq} we present the running times of these algorithms when applied to problems with $m = \Ocal \left( n^2 \right)$. Again, we can observe that our classical and quantum algorithms, w.r.t. dimension are the fastest in their respective models of computation and in particular the quantum implementation provided in Algorithm \ref{alg: IF-IPM for self-dual}
remains the fastest overall. Moreover, the gap in performance between our classical and quantum implementations
is linear in $n$.  

\begin{table}[ht]
    \centering
    \begin{tabular}{lll}
    \toprule 
    \hline
    \textbf{References}     & \textbf{Method}  & \textbf{Runtime} \\
    \hline  
    \cite{monteiro1998polynomial, nesterov1997self, nesterov1998primal}     &  IPM   & $\widetilde{\Ocal}_{n, \frac{1}{\epsilon}} \left( n^{6.5} \right)$   \\
    \cite{jiang2020faster}     &  IPM   & $\widetilde{\Ocal}_{n, \frac{1}{\epsilon}} \left( n^{5.246} \right)$  \\
    \cite{jiang2020improved, lee2015faster} & CPM    & $\widetilde{\Ocal}_{n, R, \frac{1}{\epsilon}} \left(  n^6\right)$ \\
    \cite{arora2006multiplicative} & MMWU  & $\widetilde{\mathcal{O}}_{n} \left(n^5  \left( \frac{Rr}{\epsilon}
\right)^4 + n^2 \left( \frac{Rr}{\epsilon}
\right)^7 \right)$    \\
    \cite{van2018improvements} & QMMWU   &$ \widetilde{\Ocal}_{n, \frac{1}{\epsilon}} \left( \left(n^2 + n^{1.5}
\frac{Rr}{\epsilon}\right) \left(\frac{Rr}{\epsilon}\right)^4
\right)$     \\
\cite{augustino2023quantum} & IF-QIPM  & $\widetilde{\Ocal}_{n, \kappa, \frac{1}{\epsilon}} \left(  \sqrt{n} \left(  n^{3} \kappa^2 \epsilon^{-1} + n^4 \right)\right)$  \\
\cite{augustino2023quantum} & IF-IPM  & $\widetilde{\Ocal}_{n, \kappa, \frac{1}{\epsilon}} \left(    n^{4.5} \kappa \right)$  \\
This work  & IR-IF-QIPM  &$\widetilde{\Ocal}_{n, \kappa, \frac{1}{\epsilon}} \left(  \sqrt{n} \left(  n^{3} \omega^4\kappa_T^2 + n^4 \right)\right)$  \\
This work  & IR-IF-IPM  &$\widetilde{\Ocal}_{n, \kappa, \frac{1}{\epsilon}} \left(    n^{4.5} \omega^2\kappa_T \right)$   \\
    \bottomrule
    \end{tabular}
    \caption{Total running times for classical and quantum algorithms to solve \eqref{e:SDO} with $m = \Ocal (n^2)$ and row-sparsity $s=n$. Further, $R$ is an upper bound on the trace of primal optimal solutions and $r$ is an upper bound on the $\ell_1$-norm of dual optimal solutions.}
    \label{tab:runTnsq}
\end{table}

\section{Conclusion}\label{sec: conclusion}
This paper introduces an Iterative Refinement (IR) method tailored for semidefinite optimization, leveraging limited-precision feasible SDO solvers, such as limited-precision IF-QIPMs, as a subroutine. 
Demonstrating remarkable convergence properties, our proposed IR approach showcases quadratic convergence to the optimal solution set, even if the SDO problem is degenerate and/or fails to have strict complementary optimal solutions.
However, a potential limitation of the IR method surfaces in its growing cost per iteration. 
Notably, when employing a feasible Interior Point Method (IPM) as a subroutine, the number of IPM iterations for the refining problem tends to inflate as the IR process achieves higher precision. 
This phenomenon arises from the growth of the complementarity gap in the initial solutions for the refining problems as IR iterates converge closer to the optimal solution set. 
Consequently, IPMs require more iterations to compute an inexact solution for the refining problem during the final stages of IR.

Further, we introduce two alternate variants of infeasible IR methods adaptable to infeasible oracles. 
Although these methods exhibit a quadratic reduction of the optimality gap, they concurrently reduce infeasibility at a linear rate. 
An intriguing avenue for future research is to find the initialization of IPMs in the refining problem so that the number of IPM iterations remains uniformly bounded during the IR iterations.

Expanding on our contributions, the application of IR in conjunction with an Inexact-Feasible Quantum Interior Point Method (IR-IF-QIPM) proves to be instrumental in solving SDO problems to high-precision without incurring excessive computational time as seen in previous QIPMs. 
Compared to previous QIPMs, we achieved exponential speed-up w.r.t. precision by our IR-IF-QIPM. 
Additionally, w.r.t. dimension, our IR-IF-QIPM is superior compared to classical IPMs.

One notable benefit of IR lies in its mitigation of the impact of the Newton system's condition number, where an upper bound for the condition number depends on parameters such as $\omega$, $\mu$, and $\kappa_T$. 
As the optimal solution is approached ($\mu\to0$), the condition number might tend toward infinity. 
However, IR tackles this issue by implementing early termination of the IF-QIPM steps, when $\mu$ reaches a predetermined constant, e.g., $10^{-2}$. 
Observe, that while parameters like $\omega$ and $\kappa_T$ remain constant and depend only on input data, they can be exponentially large for certain SDO problems, see e.g., the challenging SDOPs of \cite{pataki2021exponential}.

To further enhance the complexity of QIPMs, techniques such as preconditioning \cite{mohammadisiahroudi2023improvements} and regularization \cite{saunders1996solving, mohammadisiahroudi2022quantum} can be deployed, addressing constants like $\omega$ and $\kappa_T$. 
Nevertheless, additional empirical and theoretical analyses are essential to explore and realize the full potential of employing IR with other inexact algorithms, such as ADMM-based IPMs \cite{deng2022new},  Spectral Bundle Methods \cite{helmberg2000spectral}, and various others.

This paper has primarily delved into the theoretical analysis of the IR method and its integration with QIPMs. Implementations and validations have been conducted for all IR methods alongside QIPMs, and these are readily accessible through \url{https://github.com/QCOL-LU}. Furthermore, numerical findings showcased in \cite{mohammadisiahroudi2022efficient,mohammadisiahroudi2023inexact,mohammadisiahroudi2023improvements} have highlighted the efficiency of IR methods in solving challenging linear optimization problem instances, including challenging degenerate problems generated by methods outlined in \cite{mohammadisiahroudi2023generating}.

A notable avenue for further research involves benchmarking various inexact SDO solvers within the IR framework, particularly for challenging SDO problems generated by methodologies in \cite{mohammadisiahroudi2023generating}. It is crucial to note that current quantum computers lack the ability to solve linear systems. Consequently, the use of classical simulators for executing QLSAs on small linear systems is feasible. However, empirical analysis of the IR-IF-QIPM encounters limitations due to the size of the OSS system, which even for small-scale problems surpasses the capabilities of available simulators. At the time of writing this paper, empirical analysis of IR-IF-QIPMs remain unfeasible.

In conclusion, the proposed IR method serves as a bridge between exact and inexact SDO solvers, accelerating the convergence of inexact solvers while exhibiting logarithmic precision dependence akin to exact IPMs. 
This pioneering paper highlights the novel application of IR to expedite the attainment of high-precision solutions for SDOPs, emphasizing the need for further empirical and theoretical investigations to unlock the full potential of IR alongside diverse variants of inexact algorithms.

\section{Acknowledgements}
This work is supported by the Defense Advanced Research Projects Agency as part of the project W911NF2010022: {\em The Quantum
Computing Revolution and Optimization: Challenges and Opportunities}; and by the National Science Foundation (NSF) under Grant No. 2128527.

{\hyphenpenalty=100000
\bibliography{references}
}

\appendix

\section{Other SDO oracles} \label{sec: otheIR}
In this section, we develop two variants of IR method for SDO which uses an infeasible limited-precision SDO solver as a subroutine. 
\subsection{Using an Infeasible Non-Interior Oracle}

In what follows, we assume access to an oracle $O_{\text{IN}}$ which produces an approximate infeasible non-interior solution. Let $e_{\min}(A)$ denote the smallest eigenvalue of matrix $A$. The set of $\epsilon$-precise solutions is defined as
\begin{align*}
  \mathcal{PD}_{\epsilon}= \Bigg\{ (X,y,S)\in \Scal^n \times\R{m} \times \Scal^n : &~e_{\min}(X)\geq-\epsilon, \;\; e_{\min} \left(S\right) \geq-\epsilon,\\ 
  &~S=C -\sum_{i \in [m]} y_i A_i, \\
  &~\max_{i\in [m]}\{|b_i- A_i \bullet X|\}\leq \epsilon, 
  ~\; X\bullet S
  \leq \epsilon \Bigg\}.                       
\end{align*}
Hence, our oracle $O_{\text{IN}}$ returns solutions from the set $\mathcal{PD}_{\epsilon}$.

\begin{definition}\label{def:IRIN}
Let the SDO primal problem be given as \eqref{e:SDO} and \eqref{e:SDD}. Assume strong duality (zero-duality gap) holds for this problem.
For any $X \in \mathcal{S}^{n}$, $y\in \mathbb{R}^m$,  and scaling factor, $\eta>1$, the refining problem $(\bar{P})$ is defined as
\begin{equation*}
\min_{\overbar{X} \in \mathcal{S}^{n}} \left\{ \eta S\bullet \overbar{X} ~:~
    A_i\bullet \overbar{X} = \eta \bar{b}_i \; \text{\rm for }i\in[m], \;\text{\rm and } \overbar{X} \succeq -\eta X\right\},
\end{equation*} 
and its dual problem $(\bar{D})$ is as 
\begin{equation*}
\max_{(\bar{y}, \overbar{S}) \in \R{m} \times \Scal^n} \left\{  \eta \bar{b}^{\top} \bar{y}-\eta X \bullet\overbar{S} :
     \sum_{i \in [m]} \bar{y}_i A_i + \overbar{S}=\eta S, \;\text{\rm and } \overbar{S}\succeq 0\right\},
\end{equation*} 
where $S=C-\sum_{i \in [m]} y_iA_i$ and $\bar{b}_i=b_i-A_i \bullet X$ for $i\in[m]$. 
\end{definition}

Theorem \ref{theorem:iterative refinement idea IN} is the foundation of the Iterative Refinement method.
\begin{theorem} \label{theorem:iterative refinement idea IN}
Let $X \in \mathcal{S}^{n}$, $y\in \mathbb{R}^m$,  and scaling factor, $\eta>1$ and the primal-dual refining problems $(\bar{P})$ and $(\bar{D})$ are defined as in Definition~\ref{def:IRIN}. 
If $(\overbar{X},\bar{y},\bar{S})$ is an $\epsilon$-precise solutions for the refining problem $(\bar{P})$ and $(\bar{D})$, then $X+\frac{1}{\eta}\overbar{X}$ and $y+\frac{1}{\eta}\overbar{y}$ are $\frac{\epsilon}{\eta}$-precise solutions for the original SDO problem.
\end{theorem}
\begin{proof}
Since $(\overbar{X}, \overbar{y}, \overbar{S})$ is an $\epsilon$-precise solution for the refining problems, we have
{\small
\begin{align*}
    e_{\min}(\overbar{X}  +\eta X)&\geq -\epsilon,\\
    e_{\min}(\eta S- \sum_{i\in [m]} \bar{y}_iA_i) = e_{\min}(\eta C- \eta \sum_{i\in [m]} y_iA_i- \sum_{i\in [m]} \bar{y}_iA_i)&\geq -\epsilon,\\
    |\eta \bar{b}_i-A_i \bullet \overbar{X}|=|\eta b_i-\eta A_i \bullet X-A_i \bullet \overbar{X}|&\leq \epsilon,\quad   \forall i \in [m],\\
    (\overbar{X}  +\eta X) \bullet \overbar{S}&\leq\epsilon.
\end{align*}
}
For solution $X+\frac{1}{\eta}\overbar{X}$  and $y+\frac{1}{\eta}\bar{y}$, we have
{\small
\begin{align*}
    e_{\min}(X+\frac{1}{\eta}\overbar{X}) =\frac{1}{\eta}e_{\min}(\overbar{X}  +\eta X)&\geq  -\frac{\epsilon}{\eta},\\
    e_{\min}(C- \sum_{i\in [m]} (y_i+\frac{1}{\eta}\bar{y})A_i)=\frac{1}{\eta}e_{\min}(\eta C- \eta \sum_{i\in [m]} y_iA_i- \sum_{i\in [m]} \bar{y}_iA_i)  &\geq -\frac{\epsilon}{\eta},\\
    |b_i-A_i \bullet (X+\frac{1}{\eta}\overbar{X})|=\frac{1}{\eta}|\eta b_i-\eta (A_i \bullet X)- A_i \bullet\overbar{X}|&\leq \frac{\epsilon}{\eta},\quad \forall i \in [m],\\
    (X+\frac{1}{\eta}\overbar{X})\bullet(C- \sum_{i\in [m]} (y_i+\frac{1}{\eta}\bar{y})A_i)=\frac{1}{\eta^2}(\overbar{X}  +\eta X) \bullet (\eta S- \sum_{i\in [m]} \bar{y}_iA_i)&\leq\frac{\epsilon}{\eta^2}\leq \frac{\epsilon}{\eta}.
\end{align*}
}
Thus, $X+\frac{1}{\eta}\overbar{X}$  and $y+\frac{1}{\eta}\bar{y}$ are $\frac{\epsilon}{\eta}$-precise solutions for the original SDO problem.
\end{proof}
Based on Theorem \ref{theorem:iterative refinement idea IN}, we can develop an Iterative Refinement algorithm using oracle ${\rm O_{IN}}$ to improve the precision of solution for an SDO problem as follows.
\begin{algorithm}[h] 
\SetAlgoLined
\KwIn{Problem data $A_1, \dots, A_m, C \in \Scal^n$, $b \in \R{m}$, error tolerances $0 < \tilde{\epsilon} \ll \epsilon < 1$. Strictly feasible point for \eqref{e:SDO}-\eqref{e:SDD} $(X^{(0)}, y^{(0)}, S^{(0)})$\\\qquad\ \quad Choose incremental scaling limit $\rho \in \mathbb{N}$ such that $\rho > 1$} 
\KwOut{An $\tilde{\epsilon}$-precise solution pair $(X, y, S)$ to the SDO problem $(A_1, \dots, A_m, b, C).$}
\textbf{Initialize}: $\eta^{(1)} \gets 1$,  $k \gets 1$\\  
\noindent $(X^{(1)}, y^{(1)}, S^{(1)}) \gets$ \textbf{solve} \eqref{e:SDO}-\eqref{e:SDD} to precision $\epsilon$\\
\While{$X^{(k)}\bullet S^{(k)} > \tilde{\epsilon}$}{
\begin{enumerate}
\item Calculate $\bar{b}_{i}^{(k)} \gets b_i - A_i \bullet X^{(k)}~\text{for}~i \in [m]$ 
\item Calculate residual 
    $$r \gets \max \left\{\max_i |\bar{b}^{(k)}_i|,  X^{(k)}\bullet S^{(k)},\max\{-e_{\min}(X^{(k)}),0\},\max\{-e_{\min}(S^{(k)}),0\} \right\}$$
    \item Update scaling factor $\eta^{(k+1)} = \min\{\frac{1}{r},\rho \eta^{(k)}\}$
    \item $(\overbar{X}, \overbar{y}, \overbar{S}) \gets$ \textbf{solve} refining problem $(A_1, \dots, A_m, \eta^{(k)}\bar{b}^{(k)}, \eta^{(k)}S^{(k)})$ to precision $\epsilon$
    \item Update solution 
    $$ X^{(k+1)} \gets X^{(k)} + \frac{1}{\eta^{(k)}} \overbar{X}, \quad  y^{(k+1)} \gets y^{(k)} + \frac{1}{\eta^{(k)}} \overbar{y}, \quad S^{(k+1)}=C-\sum_{i\in [m]} y^{(k+1)}_iA_i$$
    \item $k\gets k+1$
    
\end{enumerate}} 
\caption{Iterative Refinement for SDO Using an Infeasible Non-Interior Oracle}
\label{alg:IR-IN}
\end{algorithm}

\subsection{Using an Infeasible Interior Oracle}

In this section, we assume that we have access to an oracle $O_{\text{II}}$ which gives a solution from the interior of the positive semidefinite cone, i.e., $X \succ 0, \ S \succ 0$, but does not satisfy the dual and primal constraints exactly. Infeasible IPMs can be used to construct such an oracle. Letting $\overbar{C}=C -\sum_{i \in [m]} y_iA_i-S$ and $\bar{b}_i=b_i- A_i \bullet X$ for $i\in [m]$, the set of $\epsilon$-precise solutions is defined as
\begin{align*}
  \mathcal{PD}_{\epsilon}= \left\{ (X,y,S)\in \Scal^n \times\R{m} \times \Scal^n : ~X\succeq0,~S\succeq0,~\|\overbar{C}\|\leq\epsilon, ~\|\bar{b}\|\leq \epsilon, ~X\bullet S\leq \epsilon \right\}.                       
\end{align*}
As we can see, this setting is different from previous IR methods since the dual constraint has an error in this setting. To account for this residual, we need to adjust the refining problems for any $X \in \mathcal{S}^{n}$, $y\in \mathbb{R}^m$,  $S \in \mathcal{S}^{n}$, and scaling factor, $\eta>1$.

\begin{definition}\label{def:IRII}
Let the SDO primal problem be given as \eqref{e:SDO} and \eqref{e:SDD}. Assume strong duality (zero-duality gap) holds for this problem.
For any $X \in \mathcal{S}^{n}$, $y\in \mathbb{R}^m$,  $S \in \mathcal{S}^{n}$, and scaling factor, $\eta>1$ consider the refining problem $(\bar{P})$
\begin{equation*}
\min_{\overbar{X} \in \mathcal{S}^{n}} \left\{ \eta (\overbar{C}+S)\bullet \overbar{X}~:~
    A_i\bullet \overbar{X} = \eta \bar{b}_i \text{ for }i\in[m], \text{ and } \overbar{X} \succeq -\eta X\right\}, 
\end{equation*} 
where $\overbar{C}=C-\sum_{i \in [m]} y_iA_i-S$ and $\bar{b}_i=b_i-A_i \bullet X$ for $i\in[m]$.
The dual form of this refining problem is derived as
\begin{equation*}
\max_{(\bar{y}, S') \in \R{m} \times \Scal^n} \left\{  \eta \bar{b}^{\top} \bar{y}-\eta X \bullet S' :
     \sum_{i \in [m]} \bar{y}_i A_i + S'=\eta \overbar{C} + \eta S \text{ and } S'\succeq 0\right\}.
\end{equation*} 
\end{definition}
To simplify the analysis, we can change the variable $\overbar{S}=S'-\eta S$, and the dual refining problem is defined as
\begin{equation*}
\max_{(\bar{y}, \overbar{S}) \in \R{m} \times \Scal^n} \left\{  \eta \bar{b}^{\top} \bar{y}-\eta X \bullet (\overbar{S}+\eta S) :
     \sum_{i \in [m]} \bar{y}_i A_i + \overbar{S}=\eta \overbar{C}  \text{ and } \overbar{S}\succeq -\eta S\right\}.
\end{equation*} 
It is easy to verify that the duality gap for these refining problems is
$$(\overbar{X}+\eta X)\bullet S'= (\overbar{X}+\eta X)\bullet(\overbar{S}+\eta S).$$

Theorem \ref{theorem:iterative refinement idea II} is a modified version of~Theorem~\ref{theorem:iterative refinement idea IN}.

\begin{theorem} \label{theorem:iterative refinement idea II}
Let $X \in \mathcal{S}^{n}$, $y\in \mathbb{R}^m$, $S \in \mathcal{S}^{n}$  and scaling factor, $\eta>1$ and the refining problem $(\bar{P})$ and $(\bar{D})$ are defined as in Definition~\ref{def:IRII}. 
If $(\overbar{X},\bar{y},\bar{S})$ is an $\epsilon$-precise solutions for the refining problem, then $X+\frac{1}{\eta}\overbar{X}$, $y+\frac{1}{\eta}\overbar{y}$, $S+\frac{1}{\eta}\overbar{S}$ and are $\frac{\epsilon}{\eta}$-precise solutions for the original SDO problem.
\end{theorem}
\begin{proof}
Since $(\overbar{X}, \bar{y},\bar{S})$ is an $\epsilon$-precise solution of the refining problem, we have
\begin{align*}
    \overbar{X}  +\eta X&\succeq 0,\\
    \bar{S}  +\eta S&\succeq 0,\\
    \|\eta \overbar{C}- \sum_{i\in [m]} \bar{y}_iA_i-\bar{S}\| = \|\eta C- \eta \sum_{i\in [m]} \tilde{y}_iA_i-\eta S- \sum_{i\in [m]} \bar{y}_iA_i-\bar{S}\|&\leq \epsilon,\\
    |\eta \bar{b}_i- A_i \bullet \overbar{X}|=|\eta b_i-\eta A_i \bullet X- A_i \bullet \overbar{X}|&\leq \epsilon,\\
    (\overbar{X}  +\eta X)\bullet(\bar{S}+\eta S)&\leq\epsilon.
\end{align*}
For solution $(X+\frac{1}{\eta}\overbar{X},y+\frac{1}{\eta}\bar{y},S+\frac{1}{\eta}\bar{S})$, we have
{\small
\begin{align*}
    X+\frac{1}{\eta}\overbar{X} =\frac{1}{\eta}(\overbar{X}  +\eta X)&\succeq 0,\\
    S+\frac{1}{\eta}\bar{S} =\frac{1}{\eta}(\bar{S}  +\eta S)&\succeq 0,\\
    \|C- \sum_{i\in [m]} (y_i+\frac{1}{\eta}\bar{y}_i)A_i-(S+\frac{1}{\eta}\bar{S})\|=  \frac{1}{\eta}\|\eta C- \eta \sum_{i\in [m]} \tilde{y}_iA_i- \sum_{i\in [m]} y_iA_i) -\eta S-\bar{S}\| &\leq \frac{\epsilon}{\eta},\\
    |b_i- A_i \bullet (X+\frac{1}{\eta}\overbar{X})|=\frac{1}{\eta}|\eta b_i-\eta A_i \bullet X- A_i \bullet \overbar{X}|&\leq \frac{\epsilon}{\eta},\\
    (X+\frac{1}{\eta}\overbar{X}) \bullet (S+\frac{1}{\eta}\overbar{S})=\frac{1}{\eta^2}(\overbar{X} +\eta X)\bullet(\overbar{S}  +\eta S)&\leq\frac{\epsilon}{\eta^2}.
\end{align*}
}
Thus, $(X+\frac{1}{\eta}\overbar{X},y+\frac{1}{\eta}\bar{y},S+\frac{1}{\eta}\bar{S})$ is an$\frac{\epsilon}{\eta}$-precise solution for the original problem.
\end{proof}
Now, we can present an iterative refinement algorithm using oracle ${\rm O_{II}}$.
\begin{algorithm}[h] 
\SetAlgoLined
\KwIn{Problem data $A_1, \dots, A_m, C \in \Scal^n$, $b \in \R{m}$, error tolerances $0 < \tilde{\epsilon} \ll \epsilon < 1$. Strictly feasible point for \eqref{e:SDO}-\eqref{e:SDD} $(X^{(0)}, y^{(0)}, S^{(0)})$\\\qquad\ \quad Choose incremental scaling limit $\rho \in \mathbb{N}$ such that $\rho > 1$} 
\KwOut{An $\tilde{\epsilon}$-precise  solution pair $(X, y, S)$ to the SDO problem $(A_1, \dots, A_m, b, C).$}
\textbf{Initialize}: $\eta^{(1)} \gets 1$,  $k \gets 1$\\  
\noindent $(X^{(1)}, y^{(1)}, S^{(1)}) \gets$ \textbf{solve} \eqref{e:SDO}-\eqref{e:SDD} to precision $\epsilon$\\
\While{$X^{(k)}\bullet S^{(k)} > \tilde{\epsilon}$}{
\begin{enumerate}
\item Calculate $$\bar{b}_{i}^{(k)} \gets b_i - A_i \bullet X^{(k)}~\text{for}~i \in [m]\quad \overbar{C}^{(k)}=C -\sum_{i \in [m]} y_i^{(k)}A_i-S^{(k)}$$ 
\item Calculate residual 
    $$r^{(k)} \gets \max \left\{\max_i |\bar{b}^{(k)}_i|,  X^{(k)}\bullet S^{(k)},\|\overbar{C}^{(k)}\| \right\}$$
    \item Update scaling factor $\eta^{(k+1)} = \min\{\frac{1}{r^{(k)}},\rho \eta^{(k)}\}$
    \item $(\overbar{X}, \overbar{y}, \overbar{S}) \gets$ \textbf{solve} refining problem $(A_1, \dots, A_m, \eta^{(k)}\bar{b}^{(k)}, \eta^{(k)}\overbar{C}^{(k)})$ to precision $\epsilon$
    \item Update solution 
    $$ X^{(k+1)} \gets X^{(k)} + \frac{1}{\eta^{(k)}} \overbar{X}, \quad  y^{(k+1)} \gets y^{(k)} + \frac{1}{\eta^{(k)}} \overbar{y}, \quad S^{(k+1)}=S^{(k)} + \frac{1}{\eta^{(k)}} \overbar{S}$$
    \item $k\gets k+1$
    
\end{enumerate}} 
\caption{Iterative Refinement for SDO Using an Infeasible Interior Oracle}
\label{alg:IR-II}
\end{algorithm}

\subsection{Complexity of the Iterative Refinement Method}
For both cases, we can prove the iteration complexity of the iterative refinement methods by using Theorem \ref{theo: convergence of iterative refinement}.
\begin{theorem} \label{theo: convergence of iterative refinement} 
The IR method returns an $\tilde{\epsilon}$-optimal solution after at most $\Ocal \left(\frac{\log(\tilde{\epsilon})}{\log({\epsilon})} \right)$ iterations. 
\end{theorem}
\begin{proof}
Based on Theorems~\ref{theorem:iterative refinement idea II} and \ref{theorem:iterative refinement idea IN} for Iterative Refinement using different oracles, after iteration $k$, the precision of the solution is $\frac{\epsilon}{\eta^{(k)}}$.
From the choice of scaling factor, we have
$$\eta^{(k)}=\min \left\{\rho\eta^{(k-1)},\frac{1}{r^{(k-1)}} \right\} \geq \eta^{(k-1)}\min \left\{\rho,\frac{1}{\epsilon}\right\}.$$
Let $\vartheta= \min \left\{\rho,\frac{1}{\epsilon}\right\}>1$, then we have
$$\eta^{(k)}\geq \vartheta \eta^{(k-1) }\geq \vartheta^k. $$
Thus, we have $\frac{\epsilon}{\eta^{(k)}}\leq\frac{\epsilon}{\vartheta^k} \leq \tilde{\epsilon}$ for 
$$k\geq \frac{\log(\epsilon/\tilde{\epsilon})}{\log(\vartheta)}\geq\frac{\log(\epsilon/\tilde{\epsilon})}{-\log(\epsilon)}.$$
We can conclude that after $\Ocal \left(\frac{\log(\tilde{\epsilon})}{\log({\epsilon})} \right)$ iterations, we have a $\tilde{\epsilon}$-optimal solution.
\end{proof}
As we can see, infeasible versions of IR converge to the optimal set linearly since the infeasibility reduces linearly, although the complementarity gap decreases quadratically. On the other hand, in infeasible IR methods, the user has more freedom to choose an inexact SDO solver as a subroutine.
\section{Analyzing the IR-IF-IPM} \label{sec: CGM}
\medskip
The IF-QIPM can be ``de-quantized'' to obtain a classical IF-IPM by replacing the use of quantum linear system solver subroutines with an inexact classical algorithm. The resulting scheme exhibits an improved condition number dependence over its quantum counterpart. Since the OSS (\ref{eq: OSS of self-dual}) system is not positive definite, this is accomplished by solving a system of the form $M^{\top}M u = M^{\top}v$ using a polynomial approximation $q$, of $1/u$ on the interval $[1/\kappa (M)^2,1]$. Proceeding in this way, the approximate solution is given by $q(M^2)Mv$, which is approximately $M^{-2}M v = M^{-1}v$. The polynomial $q$ (with degree roughly $\kappa_M \log(1/\xi)$ exists \cite[Section 6.11]{saad2003iterative}, and the approximate solution is close to $M^{-1}v$ whenever the linear system $Mu=v$ has a solution. Note that $q(M^2)M v$ is obtained using $2 \cdot \text{deg}(q)+1$ matrix-vector products, and thus avoids computing (or even writing down) the matrix $M^{\top} M$ (and is therefore preferable to the na\"ive approach of first symmetrizing the system and subsequently applying the CGM).\footnote{An explicit discussion of this method is provided in Section 16.5 of the monograph \cite{vishnoi2013lx}.} The next result from \cite{augustino2023quantum} summarizes the complexity of the IF-IPM when the Newton systems are solved using the approach we have just outlined.

\begin{theorem}

A classical implementation of the IF-IPM obtains an $\epsilon$-optimal solution $(X^*, y^*, S^*)$ to the primal-dual SDO pair \eqref{e:SDO}-\eqref{e:SDD} with overall complexity
\begin{equation*}
 \Ocal \left( n^{4.5} \kappa \cdot \textup{polylog}\left(\kappa, \frac{1}{\epsilon} \right) \right),
\end{equation*}
where $\kappa$ is an upper bound for the condition number of the condition number of the coefficient matrices of the OSS (\ref{eq: OSS of self-dual}) systems during the IF-IPM iterations.
\end{theorem}
We showed in Theorem~\ref{theo: conditionAHO} that $\kappa=\frac{\omega^2\kappa_T}{\epsilon}$ for the OSS system. Thus, the complexity of IF-IPM using CGM has a linear dependence on inverse precision. Likewise, the proposed IR method can be used to improve the complexity of classical IF-IPM with respect to precision.

When the derived IF-IPM is used in our IR methodology, one obtains an IR-IF-IPM with superior complexity bound.

\begin{theorem}
A classical implementation of the IR-IF-IPM obtains a feasible $\epsilon$ optimal solution for an SDO problem using at most $\Ocal(n^{4.5}\omega^2\kappa_{T}\log(\frac{n\mu^0}{\epsilon}))$ arithmetic operations.
\end{theorem}

\end{document}